\documentclass[final,a4paper,11pt]{article}
\usepackage{mathptmx}
\usepackage{amssymb,amsmath, amsfonts,amsthm,cases}
\usepackage{a4wide}
\usepackage{color}

\newcommand{\R}{\mathbb{R}}

\newcommand{\norm}[1]{\|#1\|}

\newcommand{\dist}[1]{{\rm dist}(#1)}

\newcommand{\mv}{\,\mid\,}

\newcommand{\K}{{\cal K}}

\renewcommand{\O}{{\mathcal O}}

\newcommand{\F}{{\cal F}}

\newcommand{\Np}{\hat{\cal N}^p}

\newcommand{\setto}[1]{\mathop{\rightarrow}\limits^{#1}}
\newcommand{\longsetto}[1]{\mathop{\longrightarrow}\limits^{#1}}
\newcommand{\skalp}[1]{\langle #1\rangle}

\newcommand{\wb}{\bar w}
\newcommand{\xb}{\bar x}
\newcommand{\yb}{\bar y}
\newcommand{\zb}{\bar z}

\newcommand{\db}{\bar d}

\newcommand{\dba}{{\bar d^\ast}}
\newcommand{\oo}{o}

\newcommand{\argmin}{{\rm arg\,min\,}}

\newcommand{\range}{{\rm Range\,}}
\newcommand{\lin}[1]{{\cal L}(#1)}

\newcommand{\Spanb}[1]{\big(#1\big)^+}

\newcommand{\gph}{\mathrm{gph}\,}
\newcommand{\dom}{\mathrm{dom}\,}
\newcommand{\tto}{\rightrightarrows}

\newcommand{\myvec}[1]{\left(\begin{array}{c}#1\end{array}\right)}

\newcommand{\KPg}{\K_P(g(\yb),\dba)}

\newcommand{\poly}{ polyhedral }

\newtheorem{theorem}{Theorem}[section]
\newtheorem{proposition}[theorem]{Proposition}
\newtheorem{remark}[theorem]{Remark}
\newtheorem{lemma}[theorem]{Lemma}
\newtheorem{corollary}[theorem]{Corollary}
\newtheorem{definition}[theorem]{Definition}
\newtheorem{example}[theorem]{Example}

\newtheorem{assumption}{Assumption}

\title{Second-order  optimality conditions for  optimization problems with  generalized equation constraints}
\author{Mat\'u\v{s} Benko\thanks{Johann Radon Institute for Computational and Applied Mathematics (RICAM), A-4040 Linz, Austria, e-mail: matus.benko@oeaw.ac.at.
This author's research was supported by the Austrian Science Fund (FWF) under grant P32832-N.}
\and Helmut Gfrerer\thanks{Johann Radon Institute for Computational and Applied Mathematics (RICAM), A-4040 Linz, Austria, and Institute of Information Theory and Automation, Czech Academy of Sciences, 18208 Prague, Czech Republic, e-mail: helmut.gfrerer@ricam.oeaw.ac.at.}
\and Jane J. Ye\thanks{Department of Mathematics
and Statistics, University of Victoria, Victoria, B.C., Canada V8W 2Y2, e-mail: janeye@uvic.ca. The research of this author was supported by NSERC.}
\and Jin Zhang\thanks{Department of Mathematics, Southern University of Science and Technology, National Center for Applied Mathematics
Shenzhen, Shenzhen 518055, P.R. China.
e-mail: zhangj9@sustech.edu.cn. }
\and Jinchuan Zhou\thanks{Department of Statistics, School of Mathematics and Statistics, Shandong University of Technology,
 Zibo 255049, P.R. China, e-mail: jinchuanzhou@163.com. The research of this author was supported by National Natural Science Foundation of China (12371305, 12571321) and Shandong Province Natural
Science Foundation (ZR2023MA020).}}

\date{}

\begin{document}

\maketitle

\noindent
 {\small
 {\bf Abstract.}\
This paper provides second-order optimality conditions for optimization problems with generalized equation constraints (GEPs), a framework that encompasses several important and challenging models in mathematical programming, including mathematical programs with variational inequality constraints (MPVIs) and bilevel programs.
The obtained optimality conditions are novel even for these particular problem classes.
As an application, second-order optimality conditions for MPVIs are detailed.
The technical key lies in developing first- and second-order variational analysis of the highly intricate constraint system,
which is needed to capture the local curvature of the feasible set entering these optimality conditions.
Part of this task was already carried out in our companion paper \cite{BeGfrYeZhangZhou}, and here we complete the study.
Comprehensive variational analysis results are derived, which are of independent interest.


\smallskip
 \noindent
 {\bf Keywords.}\ Generalized equation constraints, second-order optimality conditions, variational analysis, mathematical programs with variational inequality constraints, bilevel programming.

\smallskip
 \noindent
{\bf AMS subject classifications.}\ 49J53, 90C26, 90C33, 90C46.
}

\section{Introduction}

This paper completes our investigation initiated in \cite{BeGfrYeZhangZhou} and develops second-order optimality conditions for the class of optimization problems of the form
\begin{flalign*} 
\begin{split}
\mbox{(GEP)} \hspace{50mm} \min & \ \  f(x)\\
{\rm s.t.}
& \ \  0\in F(x)+b(x)^T N_P(g(x)),
\end{split}&
\end{flalign*}
where $P \subseteq \mathbb{R}^l$ is a convex polyhedral set, $f:\R^n\to \R$, $F:\mathbb{R}^n\rightarrow \mathbb{R}^m$,
$g:\mathbb{R}^n\rightarrow \mathbb{R}^l$, and $ b:\mathbb{R}^n\rightarrow \mathbb{R}^{l\times m}$ are twice continuously differentiable, and $N_P$ denotes the normal cone of $P$.
In the special case where $n=m=l$, $g(x)=x$ and $b(x)$ is the identity matrix, the constraint reduces to the generalized equation studied by Robinson \cite{Rob81}.
In this spirit, we refer to this problem as an optimization problem with generalized equation constraints and employ the GEP moniker.
Still with $b(x)$ equal to the identity matrix but with general $g$, one can model prominent disjunctive problems such as programs with complementarity or vanishing constraints.
Yet even more complex problems and structures can be targeted with nontrivial matrices $b(x)$, and we are particularly interested in these.
Such problems include, e.g., mathematical programs with second-order cone complementarity constraints (SOC-MPCC) (see, e.g., Outrata and Sun \cite{OutSun} and Ye and Zhou \cite{YeZhou})
or problems with variational inequality constraints and bilevel programs as detailed below.

The constraints of GEP involve both the normal cone mapping $N_P$ and the coefficient matrix $b(x)$;
this requires us to simultaneously account for the nonlinear coupling effects introduced by $b(x)$ and the geometric structure of the set-valued mapping $N_P$.
This makes the second-order theory for GEPs highly challenging and accounts for our decision to split the investigation into two papers.
As far as we know, a systematic study of second-order optimality conditions for such general problems is lacking.
Even for the particular cases mentioned above, the second-order theory remains relatively limited.

A notable advantage of the GEP model lies in providing a unified framework for describing a wide range of optimization problems.
For example, consider the following mathematical program with variational inequality constraints (MPVI):
\begin{flalign*}
\begin{split}
\mbox{(MPVI)} \hspace{46mm} \min\limits_{x,y} & \ \  f(x,y)\\
{\rm s.t.}
& \ \  \langle F(x,y),y'-y\rangle \geq 0, \ \ \ \forall y'\in \Gamma(x),
\end{split}&
\end{flalign*}
where $\Gamma(x):=\{y\,|\, \psi(x,y)\in P\}$, $P \subseteq \mathbb{R}^l$ is a convex polyhedral set,
$f:\R^{n_1\times n_2}\to \R$, $F:\mathbb{R}^{n_1\times n_2}\rightarrow \mathbb{R}^{n_2}$, and $ \psi:\mathbb{R}^{n_1\times n_2}\rightarrow \mathbb{R}^l$ are twice continuously differentiable.
Alternative names for MPVIs include generalized bilevel programming problems or mathematical programs with equilibrium constraints (see e.g., \cite{GfrYe20,lpr,MZ21,yyz,yzz}).
Particularly interesting and important is the case where $\Gamma(x)$ is convex-valued.
Then, the variational inequality constraint describes the normal cone inclusion $-F(x,y) \in N_{\Gamma(x)}(y)$
and if the constraint system $\Psi(x,y) \in P$ in variable $y$ defining $\Gamma(x)$ satisfies a constraint qualification, the MPVI becomes
\begin{align}
\hspace{34mm} \min & \ \  f(x,y) \nonumber\\
{\rm s.t.}
& \ \  0\in F(x,y)+ N_{\Gamma(x)}(y) = F(x,y)+ \nabla_y \psi(x,y)^TN_P(\psi(x,y)), \label{MPVI-1}
\end{align}
which is in the GEP form.
On the other hand, if $F(x,y):=\nabla_y \varphi (x,y)$ with $\varphi:\mathbb{R}^{n_1\times n_2}\rightarrow \mathbb{R}$ being differentiable and pseudo-convex in $y$,
the MPVI is equivalent to the following bilevel program:
  \begin{flalign*}
\begin{split}
\mbox{(BP)} \hspace{46mm} \min & \ \  f(x,y)\\
{\rm s.t.}
& \ \   y\in \argmin_ {y'}\{ \varphi(x,y') \,  | \, y'\in \Gamma(x)\}.
\end{split}&
\end{flalign*}
The reader is referred to the monographs \cite{Dempe,FPang,lpr,okz} and the references therein for background and applications of MPVIs and BPs.

To derive second-order optimality conditions for the GEP,
we reformulate it in the standard form of a general constrained optimization problem:
\begin{flalign*}
\begin{split}
\mbox{(GP)} \hspace{49mm} \min & \ \  f(x)\ \ \ {\rm s.t. } \ \  \varphi(x)\in \Omega.
\end{split}&
\end{flalign*}
To this end, we set
\begin{equation}
\label{EqOmega-1}
\varphi(x):=(x, -F(x)) \ \text{ and } \
\Omega:=\{(x,b(x)^T\eta) \,|\, (g(x),\eta) \in {\rm gph}N_P\},
\end{equation}
where ${\rm gph}N_P$ stands for the graph of the normal cone mapping $N_P$.
Most results are available for second-order conditions for the GP when $\Omega$ is convex; see the classical reference \cite{BonSh00}.
For more recent works in the convex setting, we refer to \cite{MMS21} based on modern variational analysis,
as well as \cite{AXY23} where some of the fundamental results of \cite{BonSh00} are extended to infinite-dimensional spaces
and even the smoothness requirements are relaxed.

In our case, however, $\Omega$ is rarely convex.
Results most relevant to us were initiated in \cite{GfrYeZh19}
and further developed in \cite{BeGfrYeZhangZhou,BM23,OYZ25}.
Our companion paper \cite{BeGfrYeZhangZhou} contains a deeper discussion on second-order optimality conditions for GP;
here, we limit our remarks to just mentioning that such conditions will, naturally, combine first- and second-order
derivatives of $f$ and $\varphi$ as well as the local curvature of $\Omega$, see Proposition \ref{Prop2.17}.
As in \cite{BeGfrYeZhangZhou}, we will use three variational tools to capture the curvature of $\Omega$.

The above is true for GPs with an arbitrary set $\Omega$.
If $\Omega$ is simple, its curvature can perhaps be easy to compute or even vanish altogether.
Our $\Omega$ from \eqref{EqOmega-1}, however, possesses very complicated structure.
Computing its curvature is actually precisely the key step and the major challenge of our work.
Ultimately, it will be expressed in terms of derivatives of $b$ and $g$ and the curvature of ${\rm gph}N_P$,
which is trivial due to its polyhedrality.

Instead of focusing specifically on $\Omega$ from \eqref{EqOmega-1}, we employ a slightly more general setting, targeting the structural essence of \eqref{EqOmega-1}.
To this end, we consider
 \[
 \Omega:=\{B(z)\, | \, z\in\Gamma\}\quad\mbox{ with }\quad \Gamma:=\{z\, | \, G(z)\in D\},
 \]
where $B:\R^{\hat{n}}\to \R^{\hat{m}}$ and $G:\R^{\hat{n}}\to \R^{\hat{l}}$ are twice continuously differentiable mappings, and $D$ is a polyhedral set.
The particular choices $z:=(x,\eta) $, $B(z):=(x, b(x)^T\eta)$, $G(z):=(g(x),\eta)$, and $D:={\rm gph}N_P$ indeed lead to \eqref{EqOmega-1}.
The key is to connect the curvature of $\Omega$ with the curvature of $D$ and even variational geometries of the two sets more broadly.
As expected, this is done in two main steps via the intermediate set $\Gamma$.
The core part of \cite{BeGfrYeZhangZhou} was dedicated to linking variational geometries of $\Gamma$ and $D$
and here we proceed with the second step of connecting $\Omega$ with $\Gamma$.
These fundamental results are of independent interest beyond the purposes of this paper.

The remainder of this paper is organized as follows. Section 2 reviews necessary tools from variational analysis and relevant preliminary results, including the main results of \cite{BeGfrYeZhangZhou}.
Section 3 provides a crucial and comprehensive geometric analysis of the set $\Omega=B(\Gamma)$.
Based on this, Section 4 tackles the structure of \eqref{EqOmega-1} and Section 5 establishes second-order necessary and sufficient optimality conditions for the GEPs.
As an application, Section 6 details second-order optimality conditions for a specific MPVI. 
Conclusions are drawn in Section 7.

\section{Preliminaries and auxiliary results}

In this section, we clarify the notation, recall background material from variational analysis, and present preliminary results.
The open unit ball is denoted by $\mathbb{B}$.
For a set $S \subset \R^n$, denote by  ${\rm span}\, S$ the linear subspace generated by $S$.
The indicator function $\delta_S:\R^n \to \bar \R := [-\infty,+\infty]$ of $S$
is given as $\delta_S(z)=0$ for $z \in S$ and $\delta_S(z)=+\infty$ if $z \notin S$.
Let $S^\circ$ and $\sigma_S :\R^n \to \bar \R$ stand for the polar cone of $S$ and the support function of $S$, respectively, i.e.,
 $S^\circ:=\{z^* \in \R^n \mv \langle z^*,z \rangle \leq 0,\  \forall z\in S \}$
and $\sigma_S(z^*):=\sup \{\langle z^*,z \rangle \mv z\in S\}$ for $z^*\in \R^n$.
For an extended-valued function $\varphi:\R^n \to [-\infty,\infty]$, its effective domain is ${\rm dom}\varphi:=\{z\,|\,\varphi(z) <+\infty\}$.
For $w \in \R^n$, denote by ${\{w\}}^\perp$ the orthogonal complement of the linear subspace generated by $w$.
 Let $o:\mathbb{R}_+\rightarrow \mathbb{R}^n$ stand for a mapping satisfying $o(t)/t \rightarrow 0$ as $t \downarrow 0$.
 The symbol $z^{\prime} \overset{S}{\to} z$ indicates that $z^{\prime}\in S$
 and $z^{\prime} \rightarrow z$.
   For a mapping $\psi:\mathbb{R}^n\to \mathbb{R}^d$ and a point $z \in \R^n$, we denote by
 $\nabla \psi(z)\in \mathbb{R}^{d\times n}$ its Jacobian
 and by $\nabla^2 \psi(z)$ its second derivative at $z$ as defined by
 $$w^T\nabla ^2 \psi(z) :=\lim_{t\rightarrow 0} \frac{\nabla \psi(z+ tw)-\nabla \psi( z)}{t} \quad \forall\, w \in \mathbb{R}^n.$$
 Hence
  $$\nabla^2 \psi(z)(w,w):=w^T \nabla^2 \psi(z) w=(w^T \nabla^2 \psi_1(z)w, \dots, w^T \nabla^2 \psi_d(z)w)^T \quad \forall\, w\in \mathbb{R}^n.$$

\subsection{Variational geometry}
First we review the various concepts of tangent cones and normal cones.
 \begin{definition}[Tangent Cones, \cite{Mor06,RoWe98}]
Given a closed set $C\subseteq \mathbb{R}^n$, $z\in C$ and $u\in \R^n$, the {\em tangent/contingent cone} to $C$ at $z$ is defined by
 \begin{align*}
 T_C(z):=
  \big\{u\in\R^n \, | \, \exists \ t_k\downarrow 0,\; u_k\to u \ \ {\rm with}
 \ \ z+t_k u_k\in C \big\}.
 \end{align*}
For $z\in C$ and $u\in T_C(z)$,
the outer second-order tangent set to $C$ at $z$ in direction $u$ is defined by
 \[
 T_C^2(z; u)
 := \left \{ w \in \R^n \, | \, \exists \ t_k \downarrow 0, w_k \rightarrow w \
 {\rm such \ that}\  z+t_ku+\frac{1}{2}t^2_k w_k \in C
  \right \}.
 \]
 \end{definition}

 \begin{definition}[Normal Cones, \cite{Gfr13a,GM,Mor06,RoWe98}]\label{NormalCone}
 Given a closed set $C\subseteq \mathbb{R}^n$, $z\in C$ and a direction $u\in \mathbb{R}^{n}$, the regular/Fr\'echet normal cone, the  limiting/Mordukhovich normal cone to $C$ at $z$ and the limiting normal cone to $C$ at $z$ in direction $u$ are given respectively by
 \begin{align*}
 \widehat{N}_C(z)&:=\left\{v\in \R^n \, | \, \langle v, z'-z\rangle \le
 o\big(\|z'-z\|\big), \ \forall z'\in C\right\},\\
 N_C(z)&:=
 \left \{v\in \R^n\, |\, \exists  z_k \overset{C}{\to} z, v_{k}\rightarrow v \ {\rm with }\  v_{k}\in \widehat{N}_{C}(z_k) \right \},\\
N_{C}(z; u)&:=
\left \{v\in \R^n \, |\, \exists t_{k}\downarrow 0, u_{k}\rightarrow u, v_{k}\rightarrow v \ {\rm with }\  v_{k}\in \widehat{N}_{C}(z+ t_{k}u_{k}) \right \}.
\end{align*}
\end{definition}
\if{
\begin{definition} \label{DirectionN}
Let $\wb \in \mathbb{R}^n$.  For $\delta, \rho>0$, the set
$$
 V_{\delta, \rho}(\wb) := \left \{ w\in \delta \mathbb{B}\left |  \big\|\|\wb\|w-\|w\|\wb\big\|\leq \rho \|w\|\|\wb\| \right. \right\}
$$
  is called a directional neighborhood of the direction $\wb$.
\end{definition}
}\fi

\begin{definition}[Directional Proximal Normal Cones, \cite{BeGfrYeZhangZhou}]
Given a closed set $C\subseteq \mathbb R^n$, a point $z \in C$ and a direction $u\in T_C(z)$,
we define the proximal prenormal cone to $C$ at $z$ in direction $u$ as
\[\Np_C(z;u):=\{z^*\, |\, \exists \delta,\rho, \gamma>0: \skalp{z^*,z'-z}\leq \gamma \|z'-z\|^2 \quad \forall z'\in C\cap (z+V_{\delta,\rho}(u))\} ,\]
where
\[
 V_{\delta, \rho}(u) := \left \{ u' \in \delta \mathbb{B}\left |  \big\|\|u\|u'-\|u'\|u\big\|\leq \rho \|u'\|\|u\| \right. \right\}
\]
  is called a directional neighborhood of the direction $u$.
The proximal normal cone to $C$ at $z$ in direction $u$ is given as
\[\widehat N^p_C(z;u):= \Np_C(z;u)\cap [u]^\perp.\]
In case when $u\not\in T_C(z)$ we set $\Np_C(z;u):= \widehat N^p_C(z;u):=\emptyset$.
\end{definition}
 \begin{definition}[Lower Generalized Support Function, \cite{GfrYeZh19}]\label{def-hat-sigma}
   Given a nonempty closed set $C\subseteq  \R^n$, we define the {\em lower generalized support function to $C$} as the mapping $\hat\sigma_C:\R^n\to \bar \R$ by
   \[
     \hat\sigma_C(\lambda):=\liminf_{\tilde\lambda \to\lambda}\inf_u\{\tilde  \lambda^Tu \mid \tilde\lambda\in N_C(u)\}=\liminf_{\tilde\lambda\to\lambda}\inf_u\{ \tilde\lambda^Tu \mid \tilde\lambda\in \widehat N_C(u)\}.
   \]
   If $C=\emptyset$, then we define  $\hat\sigma_C(\lambda):=-\infty $ for all $\lambda.$
 \end{definition}
 It was shown in \cite[Proposition 6]{GfrYeZh19} that in general $\hat \sigma_C(\lambda)\leq \sigma_C(\lambda)$ for all $\lambda$ and the equality holds when $C$ is convex.
\begin{definition}[Second-order Subderivative, \cite{Mor24,RoWe98}]  Let $\varphi: \mathbb{R}^n \rightarrow \bar{\mathbb{R}}$, $\varphi(\bar z)$ be finite and $\bar v\in \mathbb{R}^n$.
The second-order subderivative of $\varphi$ at $\bar z$ for $\bar v$ and $u$  is a function defined by
 \begin{eqnarray*}
 d^2\varphi(\bar z;\bar v)(u):= \liminf\limits_{{t\downarrow 0} \atop {u'\to u}} \frac{\varphi(\bar{z}+tu')-\varphi(\bar{z})-t\langle \bar{v}, u'\rangle}{\frac{1}{2}t^2}.
\end{eqnarray*}
 \end{definition}
In this paper we are mainly interested in the second subderivative of the indicator function $\delta_C$. In this case we have
\begin{eqnarray} d^2\delta_{C}(\bar{z};u^*)(u)=\liminf\limits_{{t\downarrow 0} \atop u'\to u}\frac{\delta_{C}(\bar{z}+tu')-\delta_{C}(\bar{z})-t\langle u^*, u' \rangle }{\frac{1}{2}t^2}
=\liminf\limits_{{t\downarrow 0, u'\to u} \atop { \bar z+tu'\in C}}\frac{-2\langle u^*, u' \rangle}{t}. \label{indicatorf}
 \end{eqnarray}

\begin{definition}[Directional Metric Subregularity, \cite{Gfr13a}]\label{directionMS}
Let $G:\R^n\rightarrow  \R^d$,  $ C\subseteq \R^d$ and $G(\bar z) \in C$.
 We say that the  set-valued map
 $M(z):=G(z)-C$ is metrically subregular (MS) at $(\bar{z},0)$ in direction $\wb\in \mathbb{R}^n$, if there exist
 $\kappa, \delta, \rho>0$ such that
 \[
 {\rm dist}(z,M^{-1}(0)) \leq \kappa \, {\rm dist}(G(z), C), \quad \forall z\in \bar{z}+V_{\delta,\rho}(\wb).
 \]
 In the case where  $\wb=0$, we simply say that $M$ is metrically subregular at $(\bar{z},0)$ or the metric subregularity constraint qualification (MSCQ) holds at $\bar z$.
  \end{definition}
 Sufficient conditions for directional metric subregularity are given in, e.g., \cite{BM25,Gfr11}.

\subsection{Polyhedral sets}
Given an arbitrary set $C\subseteq \R^n$, we call a subspace $L$ the {\em generalized  linearity space} of $C$ and denote it by  $\lin{C}$ provided it is the largest subspace $L\subseteq \R^n$ such that $C+L\subseteq C$. In the case where $C$ is a closed and convex cone, we have $\lin{C}=C\cap (-C)$. Further we denote by
\[C^+:={\rm span\;} (C - C)\]
the unique subspace parallel to the affine hull of $C$. In addition, if $C$ is a cone, then
$C^+={\rm span}\, C$; if $C$ is a closed and convex cone, then $C^+=C-C$. For a vector $v\in \R^n$, let $[v]$ denote the subspace generated by $v$, i.e., $[v]:=\{tv\,|\, t\in \R\}$.

A set $C\subseteq\R^n$ is said to be {\it convex polyhedral} if it is the intersection of finitely many halfspaces, whereas it is said to be {\it\poly} whenever it is the union of finitely many convex polyhedral sets. If $C$ is convex polyhedral, then
$T_C(\bar z)\subseteq T_C(z)$ for $z\in C$ near $\bar z$. But this relation fails to hold when $C$ is nonconvex polyhedral. For example,
let $C:=(\R_-,0)\cup (0,\R_+)$ and $\bar z=0$. Then $T_C(\bar z)=C$
and $T_C(z)=\{0\}\times \R$ for all $z=(0,a)\in C$ with $a>0$. So $T_C(\bar z)\nsubseteq T_C(z)$ and $T_C(z)\nsubseteq T_C(\bar z)$.
The following relations between $T_C(\bar z)$ and $T_C(z)$ will be used in the subsequent analysis.

\begin{lemma}\label{relation-tangent}
Let $C\subseteq\R^n$ be a polyhedral set and $\bar z\in C$. Then there exists an open neighborhood $V$ of $0$ such that
\[
T_C(z)=T_{T_C(\bar{z})}(z-\bar{z}), \ \ \
T_C(z)\subseteq T_C(\bar z)+[z-\bar z], \ \ \ T_C(z)^+\subseteq T_C(\bar z)^+, \ \ \ \forall \ z\in C\cap (\bar z+V).
\]
\end{lemma}
\begin{proof}
By definition $C=\displaystyle \cup_{i=1}^s C_i$ where each $C_i$ for $i=1,\dots, s$ is convex polyhedral.
Then by \cite[Theorem 2E.3]{DoRo14} there is an open neighborhood $V$ of $0$ such that
{$(C_i-\bar{z})\cap V=T_{C_i}(\bar{z})\cap V$} for $i=1,\ldots,s$.
Taking the union on both sides of the above equations, since $T_{C}(\bar{z})=\bigcup_{i=1}^s T_{C_i}(\bar{z})$, we obtain
\begin{equation} \label{Eqtangent} (C-\bar{z})\cap V=T_{C}(\bar{z})\cap V,
\end{equation}  or equivalently
\begin{equation}\label{EqRedPoly}{C \cap (\bar{z}+V)=(\bar{z}+T_C(\bar{z}))\cap (\bar{z}+V)}.\end{equation}
Hence
\[
T_C(z)=T_{\bar{z}+T_C(\bar{z})}(z)=T_{T_C(\bar{z})}(z-\bar{z}), \quad \forall z\in C\cap (\bar{z}+V).
\]
For $z\in C\cap (\bar{z}+V)$, define the active index set as $I(z):=\{i\, |\, z\in C_i\}$. Clearly
$I(z)\subseteq I(\bar z)$ by shrinking $V$ if necessary. For $i\in I(z)$, since $C_i$ is convex polyhedral, we have
\begin{align*}
T_{C_i}(z)&= T_{T_{C_i} (\bar{z})}(z-\bar{z})=[N_{T_{C_i}(\bar{z})} (z-\bar{z})]^\circ=\left [[T_{C_i}(\bar{z})]^\circ\cap (z-\bar{z})^\perp\right ]^\circ\\
&= [T_{C_i}(\bar{z})]+[(z-\bar{z})^\perp]^\circ= T_{C_i}(\bar{z})+[z-\bar{z}].
\end{align*}
Therefore
\begin{equation}\label{tangent-1}
T_C(z)=\cup_{i\in I(z)} T_{C_i}(z)= \cup_{i\in I(z)} \left(T_{C_i}(\bar{z})+[z-\bar{z}]\right)
\subseteq \cup_{i\in I(\bar{z})} T_{C_i}(\bar{z})+[z-\bar{z}]=  T_C(\bar{z})+[z-\bar{z}].
\end{equation}
Moreover by (\ref{EqRedPoly}) we have $z-\bar{z}\in T_C(\bar{z})$
implying $[z-\bar{z}]\subseteq T_C(\bar{z})^+$. This together with \eqref{tangent-1} implies
\begin{eqnarray*}
T_C(z)-T_C(z) \subseteq  T_C(\bar{z})-{T_C(\bar z)}+[z-\bar{z}]\subseteq T_C(\bar{z})^+,
\end{eqnarray*}
and hence
\[
  T_C(z)^+ \subseteq T_C(\bar{z})^+, \ \quad \forall z\in C\cap (\bar{z}+V).
\]
This completes the proof.
\end{proof}

Given a pair $(c,c^*)\in \gph \widehat N_C$, we denote by
\[\K_C(c;c^*):=T_C(c)\cap [c^*]^\perp\]
the {\em critical cone} to  $C$ at $c$ for $c^*$.

We first collect some known results for the normal cone mapping to a convex polyhedral set $P\subseteq \R^l$.
\begin{lemma}\label{LemReduction}
Let $P\subseteq \R^l$ be a convex polyhedral set and $(\bar d, \bar d^*)\in {\rm gph}N_P$. Then there is a neighborhood $V$ of $0$ such that
  \begin{align}
  T_{\gph N_P}(\db,\dba)\cap V &= \big(\gph N_P-(\db,\dba)\big)\cap V=\gph N_{\K_P(\db,\dba)}\cap V  \nonumber\\
  &= \{(e,e^*)\in \K_P(\db,\dba)\times \K_P(\db,\dba)^\circ\mv \skalp{e,e^*}=0\}\cap V. \label{eqna-a-2}
  \end{align}
  Hence
   \begin{equation}\label{equal-a-1}
   T_{\gph N_P}(\db,\dba)=\gph N_{\K_P(\db,\dba)}
  =\{(e,e^*)\in \K_P(\db,\dba)\times \K_P(\db,\dba)^\circ\mv \skalp{e,e^*}=0\}.\end{equation}
  Moreover
\begin{equation}\label{EqSecOrdTanGphN_P}T^2_{\gph N_P}\big((\db,\dba),(e,e^*)\big)= T_{\gph N_{\K_P(\db,\dba)}}(e,e^*)=\gph N_{\K_{\K_P(\db,\dba)}(e,e^*)},\ \ \ \forall \, (e,e^*)\in T_{\gph N_P}(\db,\dba).
\end{equation}
\end{lemma}
\begin{proof}
In \eqref{eqna-a-2}, the first equality follows from (\ref{Eqtangent}) since  $\gph N_P$ is polyhedral, the second one follows from the reduction lemma \cite[Lemma 2E.4]{DoRo14}, and the last equality is given by \cite[Proposition 2A.3]{DoRo14}. The formula  (\ref{equal-a-1}) is obtained by (\ref{eqna-a-2}) since the sets involved in (\ref{equal-a-1}) are all cones. Since $\gph N_P$ is a \poly set and $\K_P(\db,\dba)$ is convex polyhedral, we can obtain \eqref{EqSecOrdTanGphN_P} from \cite[Proposition 2.11]{BeGfrYeZhangZhou} and \eqref{equal-a-1}.
\end{proof}

Recall that a face of a convex set $C$ is a convex subset $C'$ of $C$ such that every (closed) line segment in $C$ with a relative interior point in $C'$ has both
endpoints in $C'$.  We denote the collection of all faces of $C$ by $\F(C)$. A convex polyhedral set $P$ has  finitely many faces which itself are convex polyhedral sets and therefore closed. For a convex polyhedral cone $K$, the faces are exactly the sets of the form $K\cap [v^*]^\perp$ for some $v^*\in K^\circ$, and there holds $\lin \F=\lin K$ for each $\F\in \F(K)$.

The following theorem is a consequence of \cite[Proof of Theorem 2]{DoRo96} and \cite[Theorem 2.12]{GfrOut16}.
\begin{theorem}\label{ThNormalConePoly}
Let $P\subseteq \R^l$ be a convex polyhedral set and $(\bar d, \bar d^*)\in {\rm gph}N_P$. Then
  \[\widehat N_{\gph N_P}(\db,\dba)=\K_P(\db,\dba)^\circ\times \K_P(\db,\dba)\]
and for every  $(e,e^*)\in T_{\gph N_P}(\db,\dba)$
the directional limiting normal cone $N_{\gph N_P}((\db,\dba); (e,e^\ast))$ is the union of all product sets $K^\circ\times K$ associated with cones $K$ of the form $\F_1-\F_2$, where $\F_1$ and $\F_2$ are closed faces of the critical cone $\K_P(\db,\dba)$ satisfying $e\in \F_2\subseteq \F_1\subseteq [e^\ast]^\perp$.
\end{theorem}

\subsection{Variational analysis of the disjunctive system}
In this section we summarize some results on first and second-order variational analysis for the disjunctive system $\Gamma:=\{z\in\R^n\mv G(z)\in D\}$ for a polyhedral set  $D\subseteq\R^l$ and a continuously differentiable mapping $G:\R^n \to\R^l$ obtained in \cite{BeGfrYeZhangZhou}.
Denote the linearization cone of $\Gamma$ at $\zb$ by
\[L_\Gamma(\zb):=\{w \in \R^n \mv\nabla G(\zb)w\in T_D(G(\zb))\}.\] Note that for every $\zb \in \Gamma$, $T_\Gamma(\zb)\subseteq L_\Gamma(\zb)$ but the reverse inclusion usually needs some assumptions.
\begin{lemma}\cite[Theorem 4.1]{BeGfrYeZhangZhou}\label{ThDirNonDegen}
 Consider a feasible point $\zb\in \Gamma$ and a direction  $\wb\in  L_\Gamma(\zb)$
    and assume that the directional nondegeneracy condition
    \begin{equation}\label{EqDirNonDegen1}\nabla G(\zb)^Tp^*=0,\ p^*\in \big(N_D(G(\zb);\nabla G(\zb)\wb)\big)^+\ \Rightarrow\ p^*=0\end{equation}
    is fulfilled. Then the set-valued map  $z\tto G(z)-D$ is metrically subregular in direction $\wb$ at $(\zb,0)$, $\wb\in T_\Gamma(\zb)$  and the following equalities hold:
    \begin{equation}\label{relation-normal}
  N_\Gamma(\zb;\wb)=\nabla G(\zb)^TN_D(G(\zb);\nabla G(\zb)\wb)=\nabla G(\zb)^T N_{T_D(G(\zb))}(\nabla G(\zb)\wb)=N_{T_\Gamma(\bar z)}(\wb).
    \end{equation}
\end{lemma}

\begin{assumption}\label{AssA2}
  The constraint mapping $z\tto G(z)-D$ is metrically subregular in direction $\wb$ at $(\zb,0)$.
\end{assumption}

Given $\bar z\in \Gamma$ and $(\wb, z^*)$, we denote the set of $S$- and $M$-multipliers associated with $(\zb, z^*)$ in direction $\wb$, respectively, by
\begin{align*}
\Lambda_{z^*}^s(\zb; \wb):=\{p^*\in \widehat N_{T_D(G(\zb))}(\nabla G(\zb)\wb)\mv z^*=\nabla G(\zb)^Tp^*\},\\
 \Lambda_{z^*}(\zb; \wb):=\{p^*\in N_{T_D(G(\zb))}(\nabla G(\zb)\wb)\mv z^*=\nabla G(\zb)^Tp^*\}.
 \end{align*}
\begin{lemma}\cite[Proposition 2.14 and Theorem 5.1]{BeGfrYeZhangZhou}\label{ThSOVAGamma}
Let $\zb\in\Gamma$ and $\wb\in L_\Gamma(\zb)$. Assume that the constraint mapping $z\tto G(z)-D$ is metrically subregular in direction $\wb$ at $(\zb,0)$. Then
\begin{equation}\label{EqSecOrdTanGamma}
T^2_\Gamma(\zb;\wb)=\{v\,|\, \nabla G(\zb)v+\nabla^2G(\zb)(\wb,\wb)\in T^2_D(G(\zb);\nabla G(\zb)\wb)\}.
\end{equation}
For all $z^*\in [\wb]^\perp$ there holds
  \[
    {\rm d^2}\delta_\Gamma(\zb;z^*)(\wb)=-\sigma_{T^2_\Gamma(\zb;\wb)}(z^*).
  \]
Moreover, $\widehat N^p_\Gamma(\zb;\wb)= \widehat N_{T_\Gamma(\zb)}(\wb)$ and for every $z^*\in \widehat N^p_\Gamma(\zb;\wb)$ we have
  \[\inf_{p^*\in \Lambda_{z^*}(\zb; \wb)}\skalp{p^*,\nabla^2 G(\zb)(\wb,\wb)}\leq {\rm d^2}\delta_\Gamma(\zb;z^*)(\wb)\leq
  \sup_{p^*\in \Lambda_{z^*}(\zb; \wb)}\skalp{p^*,\nabla^2 G(\zb)(\wb,\wb)}\]
  and
  \[
    {\rm d^2}\delta_\Gamma(\zb;z^*)(\wb)\geq\sup_{p^*\in \Lambda^s_{z^*}(\zb; \wb)}\skalp{p^*,\nabla^2 G(\zb)(\wb,\wb)}
  \]
  Finally,
   \begin{align*}
    \hat\sigma_{T^2_\Gamma(\zb;\wb)}(z^*) &\geq -\sup_{p^*\in \Lambda_{z^*}(\zb; \wb)}\skalp{p^*,\nabla^2 G(\zb)(\wb,\wb)},\ \ \ \forall z^*,\\
    \hat\sigma_{T^2_\Gamma(\zb;\wb)}(z^*) &\leq -\inf_{p^*\in \Lambda_{z^*}(\zb; \wb)}\skalp{p^*,\nabla^2 G(\zb)(\wb,\wb)},\ \ \ \forall z^*\in\dom \hat\sigma_{T^2_\Gamma(\zb;\wb)}.
  \end{align*}
 \end{lemma}

\begin{corollary}\cite[Corollary 5.8]{BeGfrYeZhangZhou}\label{CorDirNonDegen}
Let $\zb\in\Gamma$ and $\wb\in L_\Gamma(\zb)$. Under the directional nondegeneracy condition
 (\ref{EqDirNonDegen1}), the following statements hold.
 \begin{itemize}
\item[\rm (i)]  $\dom \hat\sigma_{T^2_\Gamma(\zb;\wb)}=\nabla G(\zb)^TN_{T_D(G(\zb))}(\nabla G(\zb)\wb)=N_\Gamma(\zb;\wb)$, and
 for every $z^*\in \dom \hat\sigma_{T^2_\Gamma(\zb;\wb)}$ the set $\Lambda_{z^*}(\zb; \wb)$ is a singleton $\{p_0^*\}$ and
\[
   -\hat\sigma_{T^2_\Gamma(\zb;\wb)}(z^*)=\skalp{p_0^*,\nabla^2 G(\zb)(\wb,\wb)}.
\]
 \item[\rm (ii)] $\dom \sigma_{T^2_\Gamma(\zb;\wb)}=\nabla G(\zb)^T\widehat N_{T_D(G(\zb))}(\nabla G(\zb)\wb)=\widehat N_\Gamma^p(\bar z, \bar w)$, and
for every $z^*\in \dom \sigma_{T^2_\Gamma(\zb;\wb)}$ the set $\Lambda^s_{z^*}(\zb; \wb)$ is a singleton $\{p_0^*\}$ and
 \[{\rm d^2}\delta_\Gamma(\zb;z^*)(\wb)=-\sigma_{T^2_\Gamma(\zb;\wb)}(z^*)=\skalp{p_0^*,\nabla^2 G(\zb)(\wb,\wb)}.\]
\end{itemize}
\end{corollary}

\subsection{Second-order optimality conditions for GPs}
In this section, we review second-order optimality conditions for the constrained problem GP.
Recall the following definition introduced by Penot \cite{Penot}.
\begin{definition}[Essential local minimizer of second order]
 A point $\bar{x}$ is said to be an essential local minimizer
of second order for problem
\begin{flalign*}
\begin{split}
\mbox{(GP)} \hspace{49mm} \min & \ \  f(x)\ \ \ {\rm s.t. } \ \  \varphi(x)\in \Omega,
\end{split}&
\end{flalign*}
 if there exist $\beta>0$ and $\delta>0$ such that
 \[
 \max\{f(x)-f(\bar{x}),\ {\rm dist}(\varphi(x),\Omega) \}\geq \beta \|x-\bar{x}\|^2, \ \ \ \forall x\in \mathbb{B}(\bar{x},\delta).
 \]
 \end{definition}

 For $\alpha\geq 0$, define the generalized Lagrangian function for GP as
 \[
 L^\alpha(x,\lambda):=\alpha f(x)+\langle \lambda, \varphi(x) \rangle.
 \]
 For simplicity, we omit $\alpha$  when $\alpha=1$.
 At a feasible point $\bar x$ of GP, the critical cone is defined as
 \[
 C(\bar x):=\{d\in \R^n\,|\, \nabla \varphi(\bar x)d\in T_\Omega(\varphi(\bar x)),\ \ \nabla f(\bar x)d\leq 0\}.
 \]
The following second-order optimality conditions for the general constrained optimization problem come
from \cite[Theorem 2, Corollary 5]{GfrYeZh19} and \cite[Theorem 3.3]{BeGfrYeZhangZhou}.

 \begin{proposition}[Second-order optimality conditions for GP]\label{Prop2.17}
Let $\bar{x}$ be a feasible point of problem GP. The following statements hold.
\begin{itemize}
\item[{\rm (i)}] Let $\bar{x}$ be a local optimal solution of GP. Suppose that  the set-valued map $x\tto \varphi(x)-\Omega$ is metrically subregular at $(\bar{x},0)$ in direction $d\in C(\bar x)$ with $T_\Omega^2(\varphi(\bar x); \nabla \varphi(\bar x)d)\neq \emptyset$. Then there exists a multiplier $\lambda\in \Lambda (\bar x;d) := \{\lambda \, | \, \nabla_x L(\bar x,\lambda)=0,
      \ \lambda\in {N}_{\Omega}(\varphi(\bar x);\nabla \varphi (\bar x)d)\}$ such that
\begin{equation}\label{add-second-thereom-1}
  \nabla^2_{xx}L(\bar{x},\lambda)(d,d)-\hat\sigma_{T^2_{\Omega}(\varphi(\bar x);\nabla \varphi(\bar x)d )}(\lambda)\geq 0.
\end{equation}
Moreover, if the following directional nondegeneracy condition
 \[
  \nabla \varphi(\bar{x})^T\lambda=0, \ \ \lambda\in \big(N_{\Omega}(\varphi(\bar{x});\nabla \varphi(\bar x)d)\big)^+ \Longrightarrow \lambda=0,
  \]
  holds, then $\Lambda(\bar x;d)=\{\lambda_0\}$ and
  \begin{equation*}\label{com-2-4}
  \nabla^2_{xx}L(\bar{x},{\lambda_0})(d,d)-\sigma_{T^2_{\Omega}(\varphi(\bar{x});\nabla \varphi(\bar{x})d)}({\lambda_0})
  \geq 0.
  \end{equation*}
\item[{\rm (iii)}] Suppose that for every  $d \in C(\bar x ) \backslash \{0\}$  there is some $\alpha\geq 0$ and $\lambda \in \widehat{N}_\Omega^p(\varphi(\bar x); \nabla \varphi(\bar x)d) $ not all equal to zero,
  satisfying
 \[
 \nabla_x L^\alpha (\bar x,\lambda)=0\]
  and
 \begin{equation}\label{second-general}
   \nabla^2_{xx}L^\alpha (\bar{x},\lambda)(d,d)+d^2\delta_{\Omega}(\varphi(\bar{x});\lambda)(\nabla \varphi(\bar{x})d)>0.
  \end{equation}
  Then $\bar{x}$ is an essential local minimizer of second order for GP.
\end{itemize}
 \end{proposition}

\section{Variational analysis of $\Omega$}

In this core section, we analyze the set $\Omega$ having the special representation
motivated by the common structure of optimization problems from the introduction, namely
\[\Omega:=\{B(z)\, | \, z\in\Gamma\}\quad\mbox{ with }\quad \Gamma:=\{z\, | \, G(z)\in D\}.\]
Here $B:\R^{\hat{n}}\to \R^{\hat{m}}$ and $G:\R^{\hat{n}}\to \R^{\hat{l}}$ are twice continuously differentiable maps and $D$ is a polyhedral set.
The ultimate goal is to bridge the variational geometry of $\Omega$ and that of $\Gamma$
with particular focus on second-order constructions.
Such a connection between the sets $\Gamma$ and $D$ was provided in \cite[Sections 4 and 5]{BeGfrYeZhangZhou} and reviewed
in the previous section.

One might wonder why not to proceed with an arbitrary set $\Gamma$ here.
This is of course possible, but it will soon become clear how we actually
utilize the pre-image structure of $\Gamma$ to articulate our key assumption
in a tractable and more explicit form.
Note, however, that the results are valid without any constraint qualification for the system $G(z)\in D$.

In the first part, we establish a technical foundation,
essentially securing that the mapping
$\omega \tto B^{-1}(\omega) \cap \Gamma$
locally behaves like a single-valued Lipschitzian function.
The second part contains the formulas for the tangent and normal cones,
and the third part provides those for the curvature terms associated with $\Omega$.

\subsection{The main assumption and technical foundation}

For $z\in \Gamma$, we denote by $L_z$ the linear subspace defined by
\[L_z:=\{w\,|\, \nabla G(z)w\in\Spanb{T_D(G(z))}\}.\]
Note that $T_\Gamma(z)\subseteq \{w\,|\, \nabla G(z)w\in T_D(G(z))\}$ and consequently
\begin{equation}\label{EqSpanTanCones}
  \Spanb{T_\Gamma(z)}\subseteq \Spanb{\{w\,|\, \nabla G(z)w\in T_D(G(z))\}}\subseteq \{w \,|\, \nabla G(z)w\in\Spanb{T_D(G(z))}\}=L_z.
\end{equation}
Throughout this section we assume that the following condition is fulfilled at the reference point $\bar\omega\in\Omega$.
\begin{assumption}\label{AssA1}
Let $\bar\omega\in\Omega$.
Suppose that there exists a point $\zb\in B^{-1}(\bar\omega)$ lying in the set
  \begin{align}
    \label{EqLiminfOmega}& \Sigma_{\bar \omega}:=\liminf_{\omega\setto{\Omega}\bar\omega}\big(B^{-1}(\omega)\cap \Gamma\big)
    =\big\{ \zb\, |\, \forall \omega ^k  \overset{\Omega}{\to} \bar \omega, \exists z^k \rightarrow \zb \ \ {\rm with } \ \ z^k \in B^{-1}(\omega^k)\cap \Gamma \big\}
    \end{align}
    such that the following condition holds:
    \begin{align}
    \label{EqRegB} &\nabla B(\zb)w=0,\
    w\in L_{\zb} \ \Longrightarrow \ w=0.
  \end{align}
\end{assumption}

This assumption has several important consequences.
\begin{lemma}
Let $\bar \omega\in \Omega$ and $\zb\in \Sigma_{\bar \omega}$.  For every positive real $\gamma_B$ satisfying
    \begin{equation}\label{EqGamma_B}\min\{\norm{\nabla B(\zb)w}\mv
     w\in L_{\zb} ,\norm{w}=1\}>\gamma_B,
\end{equation}
  there exists a neighborhood $\widetilde W$ of $\zb$ such that
  \begin{align}
  \label{EqRegBnb}&\norm{\nabla B(z)w}\geq \gamma_B\norm{w}, \ \ \ \forall z\in \Gamma\cap \widetilde W,\  w\in L_z,\\
  \label{EqLipHatZ}&\norm{B(z_1)-B(z_2)}\geq \gamma_B\norm{z_1-z_2}, \ \ \ \forall z_1,z_2\in \Gamma\cap \widetilde W.
  \end{align}
\end{lemma}
\begin{proof}
To the contrary, assume that  \eqref{EqRegBnb} does not hold for any neighborhood $\widetilde W$ of $\zb$. Then we can find sequences $z_k\longsetto{\Gamma}\zb$ and $w_k\in L_{z_k}$  such that $\norm{\nabla B(z_k)w_k}<\gamma_B\norm{w_k}$. Hence  $w_k\not=0$ and we can assume without loss of generality that $\norm{w_k}=1$. By possibly passing to a subsequence we can assume that $w_k$ converges to some $w$ with $\norm{w}=1$. Since $w_k\in L_{z_k}$, for $k$ sufficiently large we have
\begin{equation*}\label{wk}
\nabla G(z_k)w_k\in (T_D(G(z_k)))^+\subseteq (T_D(G(\bar z)))^+,
\end{equation*}
 where the second step follows from Lemma \ref{relation-tangent}.
 Passing to the limit and taking into account that $\norm{\nabla B(z_k)w_k}<\gamma_B\norm{w_k}$ yields $\nabla G(\zb)w\in \Spanb{T_D(G(\zb))}$, i.e., $w\in L_{\bar z}$, but $\norm{\nabla B(\zb)w}\leq\gamma_B$. This contradicts (\ref{EqGamma_B}), showing \eqref{EqRegBnb}.

  Assume now that \eqref{EqLipHatZ} does not hold. Then we can find sequences $z_1^k\longsetto{\Gamma}\zb$ and $z_2^k\longsetto{\Gamma}\zb$ such that \begin{equation}\label{B1}
  \norm{B(z_1^k)-B(z_2^k)}< \gamma_B\norm{z_1^k-z_2^k}, \ \ \forall \, k.
   \end{equation}
   Particularly, $z_1^k\not=z_2^k$. By passing to a subsequence we may assume that the sequence $(z_1^k-z_2^k)/\norm{z_1^k-z_2^k}$ converges to some $w$ with $\norm{w}=1$. Using (\ref{Eqtangent}), for all $k$ sufficiently large we have $G(z_i^k)-G(\zb)\in T_D(G(\zb))$ for $i=1,2$, 
implying $G(z_1^k)-G(z_2^k)\in \Spanb{T_D(G(\zb))}$. Since $\Spanb{T_D(G(\zb))}$ is a subspace, we obtain
  \[\nabla G(\zb)w=\lim_{k\to \infty}\frac{G(z_1^k)-G(z_2^k)}{\norm{z_1^k-z_2^k}}\in \Spanb{T_D(G(\zb))}.\]
  Hence $w\in L_{\bar z}$ and $\|w\|=1$. Further, it follows from \eqref{B1} that
\[
\norm{\nabla B(\zb)w}=\lim_{k\to\infty}\frac{\norm{B(z_1^k)-B(z_2^k)}}{\norm{z_1^k-z_2^k}}\leq \gamma_B,
\]
  contradicting \eqref{EqGamma_B}. This proves \eqref{EqLipHatZ}.
\end{proof}
\begin{proposition}\label{PropHomeo}
Let $\bar \omega\in \Omega$ and $\zb\in \Sigma_{\bar \omega}$. For every positive real $\gamma_B$ satisfying \eqref{EqGamma_B} and every open neighborhood $\widetilde W$ of $\zb$ such that \eqref{EqLipHatZ} holds, there are open neighborhoods $\O$ of $\bar\omega$ and $W$ of $\zb$ satisfying $W \subseteq\widetilde W$ and a bijective mapping $\hat z:\Omega\cap\O\to \Gamma\cap W$, such that $\hat z$ is Lipschitzian on $\Omega\cap\O$ with constant $1/\gamma_B$ and
 \[
\big(B^{-1}(\omega)\cap \Gamma \big) \cap W=\{\hat z(\omega)\},\ \ \ \forall\, \omega\in \Omega\cap\O.
\]
\end{proposition}
 \begin{proof}
   By \eqref{EqLiminfOmega} there is an open neighborhood $\O$ of $\bar\omega$ such that $Z(\omega):=B^{-1}(\omega)\cap\Gamma\cap\widetilde W\not=\emptyset$ for all $\omega\in \Omega\cap \O$. Further, for all $\omega\in \Omega\cap \O$ the set $Z(\omega)$ consists of exactly one element because for any $z_1,z_2\in Z(\omega)$ we have $\norm{z_1-z_2}\leq 1/\gamma_B\norm{B(z_1)-B(z_2)}=1/\gamma_B\norm{\omega-\omega}=0$ by \eqref{EqLipHatZ}. We denote this element by $\hat z(\omega)$. Then we have shown that $\hat z(\omega)$ is an injective mapping on $ \Omega\cap \O$.

   For every $\omega_1,\omega_2\in \Omega\cap \O$ we have $\hat z(\omega_1),\hat z(\omega_2)\in\Gamma\cap\widetilde W$ and therefore
   \[\norm{\hat z(\omega_1)-\hat z(\omega_2)}\leq \frac 1{\gamma_B}\norm{B(\hat z(\omega_1))-B(\hat z(\omega_2))}=\frac 1{\gamma_B}\norm{\omega_1-\omega_2}\]
   by \eqref{EqLipHatZ}. Thus we conclude that $\hat z$ is Lipschitzian on $\Omega\cap \O$ with constant $1/\gamma_B$. Now let $\widehat\Gamma:=\hat z(\Omega\cap \O)$. Then $\hat z:\Omega\cap \O\to\widehat\Gamma$ is bijective and, since $\hat z^{-1}=B$, $\widehat \Gamma$ is relatively open. By Assumption \ref{AssA1}, $\zb \in B^{-1}(\bar{\omega}) \cap \Gamma$ and hence $\hat z(\bar \omega)=\zb$. Together with $\zb\in \widehat\Gamma\subseteq \Gamma\cap \widetilde W$ we can find an open neighborhood $W$ of $\zb$  satisfying $W \subseteq\widetilde W$ such that {$\widehat \Gamma=\Gamma\cap W$}. This completes the proof.
 \end{proof}

\begin{remark}
Assumption \ref{AssA1} and the statement of Proposition \ref{PropHomeo} resemble the strict derivative criterion for strong metric regularity given in \cite[Theorem 4D.1]{DoRo14},
a generalization of the Kummer inverse function theorem \cite[Theorem 4D.6]{DoRo14}.
Without going into details, the local single-valuedness and Lipschitz continuity of $\omega \tto B^{-1}(\omega)\cap \Gamma$
can be derived from \cite[Theorem 4D.1]{DoRo14} under Assumption \ref{AssA1}
with \eqref{EqRegB} replaced by
\[
	\nabla B(\zb)w=0,\ w \in T^P_{\Gamma}(\bar z) \ \Longrightarrow \ w=0
\]
or by more restrictive, but explicit, condition
\[
	\nabla B(\zb)w=0,\ \nabla G(\zb)w \in T^P_D(G(\zb)) \ \Longrightarrow \ w=0.
\]
Here, $T^P$ denotes the paratingent cone, see \cite{AubFra90},
which is much larger than the tangent cone and often hard to compute.
The first condition can be applied to an arbitrary set $\Gamma$, but is hard to verify; the latter resembles ours, but $\Spanb{T_D(G(\bar z))}$ is still easier to compute.
\end{remark}

\if{
\textcolor{green}{
The following example shows that replacing $L_z$ by $\Spanb{T_\Gamma(\bar z)}$ in Assumption \ref{AssA1}
isn`t strong enough to enforce the desired behavior.
It also show that the paratingent cone is not comparable with the subspace generated by the tangent cone.}

\begin{example}
\textcolor{green}{
Let $B: \R^2 \to \R$ by given by $B(z_1,z_2) = z_1$, $G: \R^2 \to \R^2$ by $G(z_1,z_2) = (z_2 - z_1^2, z_2 + z_1^2)^T$,
$D \subset \R^2$ by $D = (\{0\} \times \R^+) \cup (\R^- \times \{0\})$ and set
$\Gamma := G^{-1} (D)$ (union of two parabolas) and $\Omega = B(\Gamma) = \R$.}

\textcolor{green}{
Clearly $B^{-1}(\omega)\cap \Gamma = \{(\omega, \omega^2), (\omega, - \omega^2)\}$ for every $\omega \in \R$.
Thus, the local single-valuedness from Proposition \ref{PropHomeo} must be violated for $\bar\omega := 0$,
for which we have $ \Sigma_{\bar \omega} = \{0\}$; we thus set $\bar z = 0$.}

\textcolor{green}{
We have $\nabla B(z)w = w_1$ for $w = (w_1,w_2) \in \R^2$ and any $z \in \R^2$.
Regarding variational geometry of $\Gamma$, we have
\[
	 \R \times \{0\} = T_\Gamma(\bar z) = \Spanb{T_\Gamma(\bar z)}
	\subset
	T^P_{\Gamma}(\bar z) = \R^2 .
\]
Particularly, $\nabla B(\bar z)w = w_1 = 0$ combined with $w \in \Spanb{T_\Gamma(\bar z)} = \R \times \{0\}$
does actually yield $w = 0$.
As expected, $\nabla B(z)w = w_1 = 0$ combined with $T^P_{\Omega}(\bar\omega) = \R^2$ allows nonzero $w$.}

\textcolor{green}{
When it comes to set $D$, we get $T_D(G(\bar z)) = D$ and
\[
	 \R^2 = \Spanb{T_D(G(\bar z))}
	\supset
	T^P_D(G(\bar z)) =   (\R \times \R^+) \cup (\R^- \times \R).
\]
Nevertheless, since $\nabla G(\zb)w = (w_2, w_2)$,
having $\nabla G(\zb)w$ in any of the two set above is the same and it imposes
no restriction on $w$, meaning that, again as expected,
Assumption \ref{AssA1} as well its modification with $T^P_D(G(\zb))$ from the remark above
are both violated.}
\end{example}
}\fi

Even more can be derived from Assumption \ref{AssA1}.
For every $z\in \Gamma$ we define the set-valued mappings $\psi_z:\R^{\hat n}\tto\R^{\hat m}\times\R^{\hat l}$ and $\theta_z:\R^{\hat m}\tto\R^{\hat n}$ by
\begin{align*}
\psi_z(w)&:=\myvec{\nabla B(z)w\\ \nabla G(z)w}-
\myvec{0_{\R^{\hat m}} \\ \Spanb{T_D(G(z))}}, \\
\theta_z(\omega^*)&:=\nabla B(z)^T\omega^*-L_z^\perp.
\end{align*}
Clearly, $\range \psi_z$ and $\range \theta_z$ are linear subspaces. Let $(\eta,q)\in \range \psi_z$.  We have
\begin{equation}\label{Eq-inverse}
 \psi_z^{-1}(\eta,q):=\left \{ w \, |\, \myvec{\nabla B(z)w\\ \nabla G(z)w} \in \myvec{\eta\\q}+\myvec{0_{\R^{\hat m}} \\ \Spanb{T_D(G(z))}}\right\}.
 \end{equation}

\begin{lemma}\label{LemPsiTheta}
Let $\bar \omega\in \Omega$ and $\zb\in \Sigma_{\bar \omega}$.  Consider $\gamma_B>0$ and a neighborhood $\widetilde W$ of $\zb$ such that \eqref{EqGamma_B}, \eqref{EqRegBnb}
and \eqref{EqLipHatZ} hold. Then for every $z\in\Gamma\cap\widetilde W$ there is a linear mapping $\hat w_z:\range \psi_z\to \R^{\hat n}$ such that $\psi_z^{-1}(\eta,q)=\{\hat w_z(\eta,q)\}$ for all $(\eta,q)\in \range \psi_z$. Moreover, $\norm{\hat w_z}$ is bounded on compact subsets of $\Gamma\cap\widetilde W$.
Further, for every $z\in\Gamma\cap\widetilde W$  the mapping $\theta_z$ satisfies
\begin{equation}\label{EqUniformMetrReg}
  \dist{\omega^*,\theta_z^{-1}(z^*)}\leq \frac 1{\gamma_B}\dist{z^*,\theta_z(\omega^*)}, \  \  \  \forall\, (\omega^*,z^*)\in\R^{\hat m}\times\R^{\hat n}.
\end{equation}
\end{lemma}
\begin{proof}
Take $z\in \Gamma\cap\widetilde W$. Then for $(\eta,q)\in \range \psi_z$ and for  $w_1,w_2\in \psi_z^{-1}(\eta,q)$, we have
\begin{align*}\myvec{\nabla B(z)(w_1-w_2)\\\nabla G(z)(w_1-w_2)}&\in \myvec{\eta\\q}+
\myvec{0_{\R^{\hat m}} \\ \Spanb{T_D(G(z))}}-\myvec{\eta\\q}-
\myvec{0_{\R^{\hat m}} \\ \Spanb{T_D(G(z))}}\\
&=\myvec{0_{\R^{\hat m}} \\ \Spanb{T_D(G(z))}}.\end{align*}
This means that $w_1-w_2\in L_{z}$ and hence $w_1=w_2$ by \eqref{EqRegBnb}.
Thus $\psi_z^{-1}(\eta,q)$ consists of a single point; let us denote it $\hat w_z(\eta,q)$. The linearity of $\hat w_z$ is an easy consequence of the linear structure of $\psi_z$. Now assume that $\norm{\hat w_z}$ is not bounded on some compact subset $M$ of $\Gamma\cap\widetilde W$. Then there are sequences $z_k\in M$ and $(\eta_k,q_k)$ with $\norm{(\eta_k,q_k)}=1$ such that $\norm{w_k}\to \infty$ with $w_k:=\hat w_{z_k}(\eta_k,q_k)$. By possibly passing to a subsequence we can assume that $z_k$ converges to some $z\in M$ and $w_k/\norm{w_k}$ converges to some $w$ with $\norm{w}=1$. By Lemma \ref{relation-tangent} we have $\Spanb{T_D(G(z_k))}\subseteq\Spanb{T_D(G(z))}$ for all $k$ sufficiently large and therefore
\[\myvec{\nabla B(z_k)\frac{w_k}{\norm{w_k}}\\\nabla G(z_k)\frac{w_k}{\norm{w_k}}}\in \myvec{\frac{\eta_k}{\norm{w_k}}\\\frac{q_k}{\norm{w_k}}}+
\myvec{0_{\R^{\hat m}} \\ \Spanb{T_D(G(z))}}.\]
Passing to the limit yields $\nabla B(z)w=0$ and $\nabla G(z)w\in \Spanb{T_D(G(z))}$, contradicting \eqref{EqRegBnb} because of $w\in L_z$ and $\norm{w}=1$. Thus, $\norm{\hat w_z}$ is bounded on $M$.

It remains to show \eqref{EqUniformMetrReg}. Since $N_{L_z^\perp}(0)=L_z$, we conclude from \eqref{EqRegBnb} that $\norm{\nabla B(z)w}\geq \gamma_B\norm{w}$ for all $w\in N_{L^\perp}(0)$. Thus, by \cite[Example 9.44]{RoWe98} the mapping $\theta_z$ is metrically regular near $(0,0)$ and hence for every $\kappa>1/\gamma_B$ there holds
\begin{equation*}
\dist{\omega^*,\theta_z^{-1}(z^*)}\leq \kappa\dist{z^*,\theta_z(\omega^*)}
\end{equation*}
for all $(\omega^*,z^*)$ sufficiently close to $(0,0)$. Due to the linear structure of $\theta_z$, we have $\theta_z^{-1}(\alpha z^*)=\alpha \theta_z^{-1}(z^*)$ and $\theta_z(\alpha\omega^*)=\alpha \theta_z(\omega^*)$
for all $\alpha>0$, yielding
\[
\alpha\dist{\omega^*,\theta_z^{-1}(z^*)}=\dist{\alpha\omega^*,\theta_z^{-1}(\alpha z^*)} \ \ {\rm and} \ \ \alpha\dist{z^*,\theta_z(\omega^*)}=\dist{\alpha z^*,\theta_z(\alpha\omega^*)}.
\]
Consequently, the above inequality is valid for all $(\omega^*,z^*)$ and since it holds for all $\kappa>1/\gamma_B$, it must hold with $\kappa=1/\gamma_B$ as well and the proof is complete.
\end{proof}

\subsection{Variational geometry of $\Omega$ via $\Gamma$}

With these preparations, the expressions of tangent cone, second-order tangent set, Fr\'echet and directional normal cone of $\Omega$ are established in Theorems \ref{ThTanConeAuxProbl} and \ref{omega-noral} below.

\begin{theorem}\label{ThTanConeAuxProbl}
Let $\bar \omega\in \Omega$ and $\zb\in \Sigma_{\bar \omega}$. Consider $\gamma_B>0$ and an open  neighborhood $\widetilde W$ of $\zb$ such that \eqref{EqGamma_B}, \eqref{EqRegBnb} and \eqref{EqLipHatZ} hold. Consider open neighborhoods $\O$ of $\bar\omega$ and $W\subseteq \widetilde W$ of $\zb$ and the bijective mapping $\hat z:\Omega\cap \O\to \Gamma\cap W$ according to Proposition \ref{PropHomeo}. Then the following statements hold true for all $\omega\in\Omega\cap\O$.
  \begin{enumerate}
  \item[\rm (i)] There holds
  \begin{equation}
    \label{EqTanAuxProbl}T_\Omega(\omega)=\{\nabla B(\hat z(\omega))w\mv w\in T_\Gamma(\hat z(\omega))\}.
    \end{equation}
 For all $\eta\in T_\Omega(\omega)$ and for all sequences $t_k\downarrow 0$ and $\eta_k\to \eta$ with $\omega+t_k\eta_k\in\Omega$,  the limit
\[ \hat z'(\omega;\eta):= \lim_{k\to\infty}\frac{\hat z(\omega+t_k\eta_k)-\hat z(\omega)}{t_k}\]
  exists and equals to $\hat w_{\hat z(\omega)}(\eta,0)\in T_\Gamma(\hat z(\omega))$, where $\hat w_z$ is the linear mapping from Lemma \ref{LemPsiTheta}.
  \item[\rm (ii)] Let $\eta\in T_\Omega(\omega)$ and $w:=\hat z'(\omega;\eta)$. Then
  \begin{equation}
    \label{EqSecOrdTanSetAuxProbl}
    T^2_\Omega(\omega;\eta)=\{\nabla B(\hat z(\omega))v+\nabla^2B(\hat z(\omega))(w,w)\mv v\in T^2_\Gamma(\hat z(\omega);w)\}.
    \end{equation}
     Further, for all $\xi\in T^2_\Omega(\omega;\eta)$ and all sequences $t_k\downarrow 0$ and $\xi_k\to\xi$ with $\omega+t_k\eta+\frac12 t_k^2\xi_k\in\Omega$, the limit
    \[\hat z''(\omega;\eta,\xi):=\lim_{k\to\infty}\frac{\hat z(\omega+t_k\eta+\frac12 t_k^2\xi_k)-\hat z(\omega)-t_k \hat z'(\omega;\eta)}{\frac 12 t_k^2}\]
    exists and equals to
  $\hat w_{\hat z(\omega)}\big(\xi-\nabla^2 B(\hat z(\omega))(w,w),-\nabla^2 G(\hat z(\omega))(w,w)\big)\in T^2_\Gamma\big(\hat z(\omega);\hat{z}'(\omega;\eta)\big).$
  \end{enumerate}
\end{theorem}
\begin{proof}
 Let $z:=\hat z(\omega)$. Then $B(z)=\omega$ by definition of the bijective mapping $\hat{z}$ in Proposition \ref{PropHomeo}. \par
 (i) We first show that the right hand side of (\ref{EqTanAuxProbl}) is included in $T_\Omega(\omega)$. For every $w\in T_\Gamma(z)$ there are sequences $t_k\downarrow 0$ and $w_k\to w$ with $z+t_kw_k\in \Gamma$. Thus $B(z+t_kw_k)\in \Omega$. Since
\[B(z+t_kw_k)=B(z)+t_k\nabla B(z){w_k}+\oo(t_k)=\omega+t_k\nabla B(z){w_k}+\oo(t_k)\in \Omega,\]
 then $\nabla B(z)w\in T_\Omega(\omega)$ follows.

 Conversely, consider $\eta\in T_\Omega(\omega)$ together with sequences $t_k\downarrow 0$ and $\eta_k\to\eta$ such that $\omega+t_k\eta_k\in\Omega$. Denoting $z_k:=\hat z(\omega+t_k\eta_k)$, we have $\omega+t_k\eta_k=B(z_k)$ by  definition of  $\hat{z}$. Hence
 \[B(z)+t_k\eta_k=B(z_k)=B(z)+\nabla B(z)(z_k-z)+\oo(\norm{z_k-z}),\] showing
 \begin{equation}\label{EqAuxEta}
\eta_k=\nabla B(z)\frac{z_k-z}{t_k}+
 \frac{\oo(\norm{z_k-z})}{t_k}.
\end{equation}
 Further since $z_k\in \Gamma$, we have $G(z_k)=G(z)+\nabla G(z)(z_k-z)+\oo(\norm{z_k-z})\in D$ and hence
 \[ \nabla G(z)(z_k-z)+\oo(\norm{z_k-z})\in T_D(G(z))\subseteq \Spanb{T_D(G(z))}\]
 for all $k$ sufficiently large, where the first step is due to \eqref{Eqtangent}. Hence
  \begin{equation} \label{EqW_z}
\nabla G(z)\frac{z_k-z}{t_k}+\frac{\oo(\norm{z_k-z})}{t_k}\in  \Spanb{T_D(G(z))}.
\end{equation}
  Combining (\ref{EqAuxEta})-(\ref{EqW_z}) together yields
 \[
 \myvec{\eta_k-\frac{\oo(\norm{z_k-z})}{t_k}\\ -\frac{\oo(\norm{z_k-z})}{t_k}}\in
 \myvec{\nabla B(z)\frac{z_k-z}{t_k} \\
 \nabla G(z)\frac{z_k-z}{t_k}} -\myvec{0_{\R^{\hat m}} \\ \Spanb{T_D(G(z))}}=\psi_z\left(\frac{z_k-z}{t_k}\right).
 \]
 Hence according to the definition of $\hat w_z$, we have
 \begin{equation}
 \frac{z_k-z}{t_k}=\hat w_z\left(\eta_k-\frac{\oo(\norm{z_k-z})}{t_k}, -\frac{\oo(\norm{z_k-z})}{t_k}\right). \label{EqDirectionderivative}
 \end{equation}
 Since $B(z_k)=\omega+t_k \eta_k$ and $B(z)=\omega$, then by (\ref{EqLipHatZ}) we have
 $\norm{z_k-z}\leq \frac{1}{\gamma_B}(\norm{B(z_k)-B(z)})=\frac{1}{\gamma_B} t_k\norm{\eta_k}$. It follows that $\oo(\norm{z_k-z})/t_k$ terms converge to $0$ as $k\to \infty$.
Taking the  limit in (\ref{EqDirectionderivative}),  we conclude that
 \[\lim_{k\to \infty}\frac{\hat z(\omega+t_k\eta_k)-\hat z(\omega)}{t_k}=\lim_{k\to\infty}\frac{z_k-z}{t_k}=\hat w_z(\eta,0).\]
 Because of $z_k\in\Gamma$ for all $k$ we have $\hat w_z(\eta,0) \in T_\Gamma(z)$ by the definition of tangent cone and from \eqref{EqAuxEta} we obtain $\eta=\nabla B(z)\hat w_z(\eta,0)$ by passing to the limit.
This proves (i).

 (ii) Fix $\eta\in T_\Omega(\omega)$ and $w= \hat z'(\omega;\eta)=\hat w_{z}(\eta,0)$ where $z=\hat z(\omega)$. By the definition of $\hat w_z$ we have $\nabla B(z)w=\eta.$
 Consider $v\in T^2_\Gamma(\hat z(\omega);w)$ together with sequences $t_k\downarrow 0$ and $v_k\to v$ such that $z_k:=z+t_k w+ \frac 12 t_k^2v_k\in \Gamma$ for all $k$. Because of
 \begin{align*}
 B(z_k)&= B(z)+t_k \nabla B(z)w+\frac 12 t_k^2\big(\nabla B(z){v_k}+\nabla ^2 B(z)(w,w)\big)+\oo(t_k^2)\\
 &=\omega + t_k\eta +\frac 12 t_k^2\big(\nabla B(z){v_k}+\nabla ^2 B(z)(w,w)\big)+\oo(t_k^2)\in\Omega,
 \end{align*}
 we readily obtain $\nabla B(z)v+\nabla ^2 B(z)(w,w)\in T^2_\Omega(\omega;\eta)$. This proves the inclusion ``$\supseteq$'' in \eqref{EqSecOrdTanSetAuxProbl}.

 Now consider $\xi\in T_\Omega^2(\omega;\eta)$ and sequences $t_k\downarrow 0$ and $\xi_k\to\xi$ with $\omega_k:=\omega +t_k\eta+\frac 12 t_k^2\xi_k\in \Omega$. Setting $z_k:=\hat z(\omega_k)$, we have $\omega_k=B(z_k)$, and hence
 \begin{align*}
 \omega +t_k\eta+\frac 12 t_k^2\xi_k &= B(z_k)=B(z)+t_k\nabla B(z)w+ \nabla B(z)(z_k-z-t_k w)+\frac 12 \nabla^2B(z)(z_k-z,z_k-z)+\oo(\norm{z_k-z}^2)\\
 &=\omega+t_k\eta+\nabla B(z)(z_k-z-t_k w)+\frac 12 \nabla^2B(z)(z_k-z,z_k-z)+\oo(\norm{z_k-z}^2).
 \end{align*}
In terms of $r_k:=(z_k-z-t_kw)/(\frac 12 t_k^2)$ and $s_k:=\nabla^2 B(z)(z_k-z,z_k-z)/t_k^2$, we obtain
 \begin{equation}\label{xi-1}
 \xi_k=\nabla B(z)r_k + s_k +\oo(\norm{z_k-z}^2)/{t_k^2}.
 \end{equation}
 Further we have
 \[G(z_k) = G(z)+t_k\nabla G(z)w + \nabla G(z)(z_k-z-t_kw)+\frac 12 \nabla^2G(z)(z_k-z,z_k-z)+\oo(\norm{z_k-z}^2)\in D,\]
 which together with \eqref{Eqtangent} and the fact that $T_D(G(z))$ is a cone implies
 \begin{equation} \nabla G(z)w+\nabla G(z)\frac{z_k-z-t_kw}{t_k}+\frac 1{2t_k}\nabla^2G(z)(z_k-z,z_k-z)+\frac{\oo(\norm{z_k-z}^2)}{t_k}\in T_D(G(z)). \label{EqG(z)}
 \end{equation}
  Since $z_k=\hat{z}(\omega_k)$ and $z=\hat{z}(\omega)$, then $B(z_k)=\omega_k$ and $B(z)=\omega$. By (\ref{EqLipHatZ}) we have
 \[\norm{z_k-z}\leq (1/\gamma_B)\norm{B(z_k)-B(z)}= (1/\gamma_B)\norm{\omega_k-\omega}=(1/\gamma_B)t_k\norm{\eta+\frac{t_k}{2}\xi_k}.\]
 So $\norm{z_k-z}/t_k$ is bounded. It further implies that $\oo(\norm{z_k-z}^2)/t_k$ and $\frac 1{2t_k}\nabla^2G(z)(z_k-z,z_k-z)$ converge to $0$ as $k\to \infty$.  Moreover by the definition of $z_k$, we have
 \begin{equation}\label{zk-z}
 \frac{z_k-z-t_kw}{t_k}=\frac{\hat z(\omega_k)-\hat z(\omega)}{t_k}-w=\frac{\hat z(\omega+t_k\eta_k)-\hat z(\omega)}{t_k}-w\to \hat z'(\omega,\eta)-w=0, \ \ {\rm as} \ \ k\to\infty,\end{equation}
  where $\eta_k:=\eta+\frac{1}{2}t_k\xi_k$ and $\omega+t_k\eta_k=\omega+t_k\eta+\frac{1}{2}t_k^2\xi_k\in \Omega$.
  Taking limits in (\ref{EqG(z)}) we obtain $\nabla G(z)w\in T_D(G(z))$. Moreover by (\ref{EqG(z)}) we have
 \begin{align*}
 \nabla G(z)r_k+ q_k+\frac{\oo(\norm{z_k-z}^2)}{t_k^2}&\in \frac{1}{(1/2)t_k}(T_D(G(z))-\nabla G(z)w)\subseteq
 \frac{1}{(1/2)t_k}\big(T_D(G(z))-T_D(G(z))\big) \nonumber\\
 &\subseteq \Spanb{T_D(G(z))}, 
 \end{align*}
 where $q_k:=\nabla^2 G(z)(z_k-z,z_k-z)/t_k^2$. This, together with \eqref{xi-1}, yields
 \[
 \myvec{\xi_k- s_k -\oo(\norm{z_k-\zb}^2)/{t_k^2}\\
 -q_k - \frac{\oo(\norm{z_k-z}^2)}{t_k^2} }\in \myvec{\nabla B(z)r_k\\ \nabla G(z)r_k}-\myvec{0_{\R^{\hat m}}\\ \Spanb{T_D(G(z))}}=\psi_z(r_k).
 \]
 By the definition of $\hat w_z$, we have $r_k=\hat w_z\left(\xi_k-s_k-\frac{\oo(\norm{z_k-z}^2)}{t_k^2},
 -q_k-\frac{\oo(\norm{z_k-z}^2)}{t_k^2}\right)$. Taking into account $(z_k-z)/t_k\to w$ by \eqref{zk-z}, we obtain
 \begin{align}
 \hat z''(\omega;\eta,\xi)&=\lim_{k\to \infty}\frac{\hat z(\omega+t_k\eta+\frac{1}{2}t_k^2\xi_k)-\hat z(\omega)-t_k\hat z'(\omega;\eta)}{\frac{1}{2}t_k^2}=\lim_{k\to \infty}\frac{z_k-z-t_kw}{\frac{1}{2}t_k^2} \nonumber\\
 &= \lim_{k\to\infty}r_k=\lim_{k\to\infty}
 \hat w_z\left(\xi_k-s_k-\frac{\oo(\norm{z_k-z}^2)}{t_k^2},
 -q_k-\frac{\oo(\norm{z_k-z}^2)}{t_k^2}\right) \label{r-k}\\
 &= \hat w_z(\xi-\nabla^2 B(z)(w,w), -\nabla ^2G(z)(w,w)). \nonumber
 \end{align}
By the definition of $\hat w_z$ and (\ref{Eq-inverse}), we obtain from the above equality that
 \begin{equation}\label{bzz}
 \nabla B(z)\hat z''(\omega;\eta,\xi)=\xi-\nabla^2 B(z)(w,w).\end{equation}
  Since $z_k=\hat z(\omega_k)\in \Gamma$,  $r_k=(z_k-z-t_kw)/(\frac{1}{2}t_k^2)\in (\Gamma-z-t_kw)/(\frac{1}{2} t_k^2)$ with $w=\hat z'(\omega;\eta)\in T_\Gamma(z)$ by (i) and $r_k\to \hat z''(\omega;\eta,\xi)$ by \eqref{r-k}, we readily obtain that $\hat z''(\omega;\eta,\xi)\in T^2_\Gamma(z;w)$.
 This, together with \eqref{bzz}, means that $\xi $ is in the right hand side of \eqref{EqSecOrdTanSetAuxProbl} and completes the proof.
\end{proof}

\begin{theorem}\label{omega-noral}
Under the assumptions of Theorem \ref{ThTanConeAuxProbl}, for every $\omega\in\Omega\cap \O$ we have
\begin{equation}\label{EqRegNormalAuxProbl}
  \widehat N_\Omega(\omega)=\{\omega^*\mv \nabla B(\hat z(\omega))^T\omega^*\in \widehat N_\Gamma(\hat z(\omega))\},
\end{equation}
 where $\hat z(\omega)$ is defined in Proposition \ref{PropHomeo}. Further, for every $\eta\in T_\Omega(\omega)$ there holds
\begin{equation}\label{EqLimNormalAuxProbl}
  N_\Omega(\omega;\eta)=\{\omega^*\mv \nabla B(\hat z(\omega))^T\omega^*\in N_\Gamma(\hat z(\omega); \hat z'(\omega;\eta))\},
\end{equation}
where $\hat z'(\omega;\eta)$ is defined in Theorem \ref{ThTanConeAuxProbl}.
\end{theorem}
\begin{proof}Let $z:=\hat z(\omega)$.
 Using \eqref{EqTanAuxProbl} we have $\omega^*\in \widehat N_\Omega(\omega)$ if and only if
 \[\skalp{\omega^*, \nabla B(z)w}=\skalp{\nabla B(z)^T\omega^*,w}\leq 0, \ \ \ \forall w\in T_\Gamma(z),\]
 i.e., $\nabla B(z)^T\omega^*\in\widehat N_\Gamma(z)$. This proves \eqref{EqRegNormalAuxProbl}.

 Now consider $\eta\in T_\Omega(\omega)$ and $\omega^*\in N_\Omega(\omega;\eta)$ together with sequences $t_k\downarrow 0$, $\eta_k\to\eta$ and $\omega_k^*\to \omega^*$ such that
 $\omega_k^*\in\widehat N_\Omega(\omega+t_k\eta_k)$. Setting $z_k:=\hat z(\omega+t_k\eta_k)$, by (\ref{EqRegNormalAuxProbl}) we have $\nabla B(z_k)^T\omega_k^*\in\widehat N_\Gamma(z_k)=\widehat N_\Gamma(z+t_kw_k)$ with $w_k:=(z_k-z)/t_k$. Since
 \[
 \lim_{k\to\infty} w_k=
 \lim_{k\to\infty}\frac{z_k-z}{t_k}=\lim_{k\to \infty}\frac{\hat z(\omega+t_k\eta_k)-\hat z(\omega)}{t_k}=\hat z'(\omega;\eta)\in T_\Gamma(z)\] by Theorem \ref{ThTanConeAuxProbl}(i), we obtain
 \[\nabla B(z)^T\omega^*=\lim_{k\to\infty}
 \nabla B(z_k)^T\omega_k^*\in N_\Gamma(z;\hat z'(\omega;\eta)).\]
 This proves the inclusion ``$\subseteq$'' in \eqref{EqLimNormalAuxProbl}.

 In order to prove the reverse inclusion, consider $\omega^*$ such that $\nabla B(z)^T\omega^*=z^*$ with $z^*\in N_\Gamma(z;w)$, where $w:=\hat z'(\omega;\eta)=\hat w_{\hat z(\omega)}(\eta,0)$. Hence
 $\nabla B(z)w=\nabla B(\hat z(\omega))w=\eta$ by the definition of $\hat w_{\hat z(\omega)}$ given in Lemma \ref{LemPsiTheta}. Since $z^*\in N_\Gamma(z;w)$, we can find sequences $t_k\downarrow 0$ and $(z_k^*,w_k)\to (z^*,w)$ such that $z+t_kw_k\in \Gamma$ and $z_k^*\in\widehat N_\Gamma(z+t_kw_k)$ for all $k$. Let $z_k:=z+t_kw_k$.
 For every $k$ since $\theta_{z_k}^{-1}(z_k^*)$ is closed, we can find
  $\omega_k^*\in\theta_{z_k}^{-1}(z_k^*)$ by Lemma \ref{LemPsiTheta}
  such that
   \begin{equation}\label{omega-lim}
   \norm{\omega_k^*-\omega^*}\leq\frac 1{\gamma_B}\dist{z_k^*,\theta_{z_k}(\omega^*)}\leq \frac 1{\gamma_B}\norm{z_k^*-\nabla B(z_k)^T\omega^*}\to 0 \ \ \mbox{ as } \ \ k\to\infty.
   \end{equation}
   Since $\omega_k^*\in\theta_{z_k}^{-1}(z_k^*)$, then
   $z^*_k\in \theta_{z_k}(\omega_k^*)=\nabla B(z_k)^T\omega_k^*-L_{z_k}^\perp$, and hence
   \[\nabla B(z_k)^T\omega_k^*-z_k^*=:w_k^*\in L_{z_k}^\perp\subseteq \Big(\Spanb{T_\Gamma(z_k)}\Big)^\perp,\]
   where the last inclusion follows by taking orthogonal complement on both sides of (\ref{EqSpanTanCones}). It follows that $\skalp{w_k^*,v}=0$ for all $v\in T_\Gamma(z_k)$ and therefore $w_k^* \in (T_\Gamma(z_k))^\circ=\widehat N_\Gamma(z_k)$. Since $z_k^*\in\widehat N_\Gamma(z+t_kw_k)=\widehat N_\Gamma(z_k)$, we have
 $\nabla B(z_k)^T\omega_k^*=z_k^*+
w_k^*\in \widehat N_\Gamma(z_k)$ because the regular normal cone is a convex cone. From \eqref{EqRegNormalAuxProbl} we conclude $\omega_k^*\in\widehat N_\Omega(B(z_k))$. Since
\[B(z_k)=B(z+t_kw_k)=B(z)+t_k\nabla B(z)w+\oo(t_k)=\omega +t_k\eta+\oo(t_k)\in \Omega,\]
then $\omega_k^*\in\widehat N_\Omega(\omega +t_k\eta+\oo(t_k))$. This, together with $\omega_k^*\to \omega^*$ by \eqref{omega-lim}, implies
$\omega^*\in N_\Omega(\omega;\eta)$.
\end{proof}

\subsection{Formulas for the curvature terms}

The following result gives the formula of second-order subderivative $d^2\delta_{\Omega}$, lower generalized support function $\hat\sigma_{T_\Omega^2}$, and support function $\sigma_{T_\Omega^2}$, which capture the local curvature of $\Omega$ appearing in second-order optimality conditions.

\begin{theorem}
  Under the assumptions of Theorem \ref{ThTanConeAuxProbl}, for every $\omega\in\Omega\cap \O$, $\eta\in T_\Omega(\omega)$, and $\omega^*\in \R^{\hat m}$, we have
  \begin{align}\label{EqSecOrdSubDerOmega}&{\rm d^2}\delta_\Omega(\omega;\omega^*)(\eta)= {\rm d^2}\delta_\Gamma(z;\nabla B(z)^T\omega^*)(w)-\skalp{\omega^*,\nabla^2 B(z)(w,w)},\\
  \label{EqSuppOmega}&\sigma_{T^2_\Omega(\omega;\eta)}(\omega^*)=\sigma_{T^2_\Gamma(z;w)}(\nabla B(z)^T\omega^*)+\skalp{\omega^*,\nabla^2 B(z)(w,w)},\\
  \label{EqLowerSuppOmega}&\hat \sigma_{T^2_\Omega(\omega;\eta)}(\omega^*)=\hat \sigma_{T^2_\Gamma(z;w)}(\nabla B(z)^T\omega^*)+\skalp{\omega^*,\nabla^2 B(z)(w,w)},
  \end{align}
  where $z:=\hat z(\omega)$ is defined in Proposition \ref{PropHomeo} and $w:=\hat z'(\omega;\eta)$ is defined in Theorem \ref{ThTanConeAuxProbl}.
\end{theorem}
\begin{proof}
(i) First we prove the equality (\ref{EqSecOrdSubDerOmega}).
  By (\ref{indicatorf}), consider  sequences $t_k\downarrow 0$ and $\eta_k\to\eta$ with $\omega_k:=\omega+t_k\eta_k\in \Omega$ satisfying
 ${\rm d^2}\delta_\Omega(\omega;\omega^*)(\eta)=\displaystyle \lim_{k\to\infty}-\frac{2\skalp{\omega^*,\eta_k}}{t_k}$. Let $z_k:=\hat z(\omega_k)\in \Gamma$ and $w_k:=(z_k-z)/t_k$. Since $\omega_k=B(z_k)$ and $\omega=B(z)$, then it follows from Theorem \ref{ThTanConeAuxProbl}(i) that
  \[\lim_{k\rightarrow \infty} w_k=\lim_{k\to\infty}\frac{z_k-z}{t_k}=
  \lim_{k\to\infty}\frac{\hat z(\omega_k)-\hat z(\omega)}{t_k}=\lim_{k\to \infty} \frac{\hat z(\omega+t_k\eta_k)-\hat z(\omega)}{t_k}=
 \hat{z}'(\omega;\eta)=w.\]
 From
 \[t_k\eta_k=\omega_k-\omega=B(z_k)-B(z)=\nabla B(z)t_k w_k+\frac 12 t_k^2\nabla^2 B(z){(w_k,w_k)}+\oo(t_k^2),\]
 we deduce
  \begin{align*}{\rm d^2}\delta_\Omega(\omega;\omega^*)(\eta)&=\lim_{k\to\infty}-\frac{2\skalp{\omega^*,\eta_k}}{t_k}=\lim_{k\to\infty}-\frac{2\skalp{\nabla B(z)^T\omega^*,w_k}}{t_k}-\skalp{\omega^*,\nabla^2 B(z)(w,w)}\\
  &=\lim_{k\to\infty}\frac{\delta_\Gamma(z+t_kw_k)-\delta_\Gamma(z)-t_k\langle \nabla B(z)^T\omega^*, w_k\rangle}{\frac{1}{2}t_k^2}-\skalp{\omega^*,\nabla^2 B(z)(w,w)} \\
    &\geq {\rm d^2}\delta_\Gamma(z;\nabla B(z)^T\omega^*)(w)-\skalp{\omega^*,\nabla^2 B(z)(w,w)}.
  \end{align*}

  In order to prove the reverse  inequality of (\ref{EqSecOrdSubDerOmega}),  by (\ref{indicatorf}), choose sequences $t_k\downarrow 0$ and $w_k\to w$ with $z_k:=z+t_kw_k\in\Gamma$ satisfying ${\rm d^2}\delta_\Gamma(z;\nabla B(z)^T\omega^*)(w)=\lim_{k\to\infty}-2\skalp{\nabla B(z)^T\omega^*,w_k}/t_k$. Since $w=\hat z'(\omega,\eta)=\hat w_{\hat z(\omega)}(\eta,0)=\hat w_z(\eta,0)$ by Theorem \ref{ThTanConeAuxProbl}(i), then $\nabla B(z)w=\eta$ according to the definition of $\hat w_z$. Let $\eta_k:=(B(z_k)-B(z))/t_k$. Then $\omega+t_k\eta_k=B(z)+t_k\eta_k=B(z_k)\in \Omega$. Taking into account $B(z_k)=B(z)+t_k\nabla B(z)w_k+\frac 12 t_k^2\nabla ^2B(z){(w_k,w_k)}+\oo(t_k^2)$, we obtain
  \begin{align*}{\rm d^2}\delta_\Omega(\omega;\omega^*)(\eta)&\leq \liminf_{k\to\infty}-\frac{2\skalp{\omega^*,\eta_k}}{t_k}
  =\liminf_{k\to\infty}-\frac{2\skalp{\omega^*, \nabla B(z)w_k}}{t_k}- \skalp{\omega^*,\nabla^2B(z)(w,w)}\\
  &={\rm d^2}\delta_\Gamma(z;\nabla B(z)^T\omega^*)(w)-\skalp{\omega^*,\nabla^2B(z)(w,w)}
  \end{align*}
  and \eqref{EqSecOrdSubDerOmega} is proved.

  (ii) The formula \eqref{EqSuppOmega} follows from \eqref{EqSecOrdTanSetAuxProbl} because of
  \begin{align*}\sigma_{T^2_\Omega(\omega;\eta)}(\omega^*)&=\sup\{\skalp{\omega^*,\nabla B(z)v+\nabla^2B(z)(w,w)}\mv v\in T^2_\Gamma(\hat z(\omega);w)\}\\
  & =\sup\{ \skalp{v, \nabla B(z)^T \omega^*}  \mv v\in T^2_\Gamma(\hat z(\omega);w)  \}+ \skalp{\omega^*,\nabla^2B(z)(w,w)} \\
  & = \sigma_{T^2_\Gamma(z;w)}(\nabla B(z)^T\omega^*)+\skalp{\omega^*,\nabla^2B(z)(w,w)}.
  \end{align*}

(iii)  We now prove \eqref{EqLowerSuppOmega}. By virtue of Theorem \ref{ThTanConeAuxProbl}(ii) we have that either both $T^2_\Omega(\omega;\eta)$ and $T^2_\Gamma(z;w)$ are not empty or both sets are empty, where in the latter case both sides in \eqref{EqLowerSuppOmega} evaluate to $-\infty$ by definition. We now assume that both sets $T^2_\Omega(\omega;\eta)$ and $T^2_\Gamma(z;w)$ are not empty. We prove first the inequality `$\geq$' in \eqref{EqLowerSuppOmega}. If $\hat\sigma_{T^2_\Omega(\omega;\eta)}(\omega^*)=\infty$ this inequality certainly holds and therefore we can assume $\hat\sigma_{T^2_\Omega(\omega;\eta)}(\omega^*)<\infty$. Then
  by definition we can find a sequence $(\xi_k,\omega_k^*)\in\gph \widehat N_{T^2_\Omega(\omega;\eta)}$ such that $\omega_k^*\to\omega^*$ and $\hat\sigma_{T^2_\Omega(\omega;\eta)}(\omega^*)=\lim_{k\to\infty}\skalp{\omega_k^*,\xi_k}$.
  For $\xi_k\in T_\Omega^2(\omega;\eta)$, we define $v_k:=\hat z''(\omega;\eta,\xi_k)$ as in Theorem \ref{ThTanConeAuxProbl}(ii). Hence
 \[
 v_k=\hat z''(\omega;\eta,\xi_k)= \hat w_{z}(\xi_k-\nabla^2 B(z)(w,w),-\nabla^2 G(z)(w,w))\in T^2_\Gamma(z;w).
 \]
 By the definition of $\hat w_z$, we obtain from the above that
 $\nabla B(z)v_k=\xi_k-\nabla^2B(z)(w,w),$
 i.e.,
 \[\xi_k=\nabla B(z)\nu_k+\nabla^2 B(z)(w,w).\]
 For $v\in T^2_\Gamma(z;w)$, let $\xi:=\nabla B(z)v+\nabla^2 B(z)(w,w).$ Then $\xi\in T_\Omega^2(\omega;\eta)$ by \eqref{EqSecOrdTanSetAuxProbl} and $\xi-\xi_k=\nabla B(z)(v-v_k)$. Since $\omega_k^* \in \widehat N_{T^2_\Omega(\omega;\eta)}(\xi_k)$, we obtain
  \begin{align*}
  \skalp{\nabla B(z)^T\omega_k^*,v-v_k}&=\skalp{\omega_k^*, \nabla B(z)(v-v_k)}=\skalp{\omega_k^*, \xi-\xi_k}
   \leq \oo(\norm{\xi-\xi_k})\\
   & = \oo(\norm{\nabla B(z)(v-v_k)})\leq \oo(\norm{v-v_k}),
  \end{align*}
  implying $\nabla B(z)^T\omega_k^*\in \widehat N_{T^2_\Gamma(z;w)}(v_k)$. Hence   \begin{align*}\hat\sigma_{T^2_\Omega(\omega;\eta)}(\omega^*)&=\lim_{k\to\infty}\skalp{\omega_k^*,\xi_k}=\lim_{k\to\infty}\skalp{\omega_k^*,\nabla B(z)v_k+\nabla^2B(z)(w,w)}\\
  &= \lim_{k\to\infty}\skalp{\nabla B(z)^T \omega_k^*,v_k}+\skalp{\omega^*,\nabla^2B(z)(w,w)}\\
  &\geq  \hat\sigma_{T^2_\Gamma(z;w)}(\nabla B(z)^T\omega^*)+\skalp{\omega^*,\nabla^2 B(z)(w,w)}.
  \end{align*}

  To show the inequality `$\leq$' in \eqref{EqLowerSuppOmega}, assume that $\hat\sigma_{T^2_\Gamma(z;w)}(\nabla B(z)^T\omega^*)<\infty$ and consider sequences $(v_k,z_k^*)$ with $z_k^*\in \widehat N_{T^2_\Gamma(z;w)}(v_k)$,  $z_k^*\to \nabla B(z)^T\omega^*$ and $\hat\sigma_{T^2_\Gamma(z;w)}(\nabla B(z)^T\omega^*)=\lim_{k\to\infty}\skalp{z_k^*,v_k}$. It is well known that
  \begin{equation}\label{tangent-set}
  T^2_\Gamma(z;w)\subseteq\{v\mv \nabla G(z)v+\nabla^2G(z)(w,w)\in T^2_D(G(z);\nabla G(z)w)=T_{T_D(G(z))}(\nabla G(z)w)\},
  \end{equation}
  see, e.g., \cite[Proof of Proposition 3.33]{BonSh00}. Since $T_{T_D(G(z))}(\nabla G(z)w)=T_{T_D(G(z))}(\alpha\nabla G(z)w)$ for all $\alpha>0$, then
  $\Spanb{T_{T_D(G(z))}(\nabla G(z)w)}\subseteq \Spanb{T_{T_D(G(z))}(0)}=\Spanb{T_D(G(z))}$ by Lemma \ref{relation-tangent}. Hence taking into account \eqref{tangent-set}, we have
  \begin{align*}
  T_\Gamma^2(z;w)-T_\Gamma^2(z;w)& \subseteq \big\{v\,|\, \nabla G(z)v\in T_{T_D(G(z))}(\nabla G(z)w)-T_{T_D(G(z))}(\nabla G(z)w) \big\}\\
  & \subseteq \big\{v\,|\, \nabla G(z)v\in \big(T_{T_D(G(z))}(\nabla G(z)w)\big)^+ \big\}\\
  & \subseteq \big\{v\,|\, \nabla G(z)v\in \big(T_D(G(z))\big)^+ \big\},
  \end{align*}
  which together with \eqref{EqSpanTanCones} yields
  \begin{equation}
  \Spanb{T^2_\Gamma(z;w)}\subseteq\{v\,|\, \nabla G(z)v\in \Spanb{T_D(G(z))}\}\subseteq L_z.\label{Eq(33)}
  \end{equation}
  Since $\theta_{z}^{-1}(z_k^*)$ is closed, by
   Lemma \ref{LemPsiTheta} we can find some $\omega_k^*\in\theta_{z}^{-1}(z_k^*)$ such that
\begin{equation} \norm{\omega_k^*-\omega^*}\leq\frac 1{\gamma_B}\dist{z_k^*,\theta_{z}(\omega^*)}\leq \frac 1{\gamma_B}\norm{z_k^*-\nabla B(z)^T\omega^*}\to 0\ \ \mbox{ as }\ \ k\to\infty. \label{Eqtwoineqs}
\end{equation}
Since $z_k^*\in \theta_{z}(\omega_k^*)$, by definition of $\theta_{z}$ we have $\nabla B(z)^T\omega_k^*-z_k^*=:w_k^*\in L_{z}^\perp\subseteq \big(\Spanb{T^2_\Gamma(z;w)}\big)^\perp,$ where the last inclusion follows from (\ref{Eq(33)}).
It follows that $\skalp{w_k^*,v-v_k}=0$ for all $v\in T^2_\Gamma(z;w)$ and therefore $w_k^* \in  \widehat N_{T^2_\Gamma(z;w)}(v_k)$, implying
 \begin{equation}\label{regular-gamma}
 \nabla B(z)^T\omega_k^*=z_k^*+w_k^*\in \widehat N_{T^2_\Gamma(z;w)}(v_k).
  \end{equation}
  By (\ref{Eqtwoineqs}), we know $\omega_k^*\to \omega^*$ and hence
 $w^*_k=\nabla B(z)^T\omega^*_k-z^*_k\to \nabla B(z)^T\omega^*-\nabla B(z)^T\omega^*=0$.
 Thus
  \begin{equation}\label{w-v}
  \skalp{w_k^*,v_k}=\skalp{w_k^*,v}\to 0, \ \ \ {\rm as} \ \ k\to \infty.
  \end{equation}
Take $\xi\in T^2_\Omega(\omega;\eta)$ and let $\xi_k:=\nabla B(z)v_k+\nabla^2 B(z)(w,w)$. Since $v_k\in  T^2_\Gamma(z;w)$ we have $ \xi_k \in T^2_\Omega(\omega;\eta)$ by (\ref{EqSecOrdTanSetAuxProbl}). According to Theorem \ref{ThTanConeAuxProbl}
\begin{align*}
\hat z''(\omega;\eta,\xi)&= \hat w_{z}(\xi-\nabla^2 B(z)(w,w),-\nabla^2 G(z)(w,w))\in T^2_\Gamma(z;w),\\
\hat z''(\omega;\eta,\xi_k)&= \hat w_{z}(\xi_k-\nabla^2 B(z)(w,w),-\nabla^2 G(z)(w,w))\in T^2_\Gamma(z;w),
\end{align*}
and hence
\begin{equation}\label{z-z}
\hat z''(\omega;\eta,\xi)-\hat z''(\omega;\eta,\xi_k)=\hat w_{z}(\xi-\xi_k,0),
\end{equation}
since $\hat w_z$ is a linear mapping by Lemma \ref{LemPsiTheta}.
 By the definition of $\hat w_z$, we obtain from the above  that
 \begin{align}
 \nabla B(z)\hat z''(\omega;\eta,\xi)&= \xi -\nabla^2 B(z)(w,w),\label{zv}\\
 \nabla B(z)\hat z''(\omega;\eta,\xi_k)&= \xi_k -\nabla^2 B(z)(w,w). \label{vk}
 \end{align}
 Recalling $\xi_k=\nabla B(z)v_k+\nabla^2 B(z)(w,w)$, then $\nabla B(z)\hat z''(\omega;\eta,\xi_k)=\nabla B(z)v^k$ by \eqref{vk}.
 Since $v^k\in T^2_\Gamma(z;w)$ and $\hat z''(\omega;\eta,\xi_k)\in T^2_\Gamma(z;w)$, then
 $v^k-\hat z''(\omega;\eta,\xi_k)\in \Spanb{T^2_\Gamma(z;w)}\subseteq L_z$ by \eqref{Eq(33)}. It follows from \eqref{EqRegBnb} that
 $\norm{v^k-\hat z''(\omega;\eta,\xi_k)}\leq 1/( \gamma_B)\norm{\nabla B(z)(v^k-\hat z''(\omega;\eta,\xi_k))}=0$. So $v_k=\hat z''(\omega;\eta,\xi_k)$.
 Hence for $\xi\in  T^2_\Omega(\omega;\eta)$, we have
 \begin{align*}
 \skalp{\omega_k^*,\xi-\xi_k}&= \skalp{\omega_k^*,\nabla B(z)\big(\hat z''(\omega;\eta,\xi)-\hat z''(\omega;\eta,\xi_k)\big)}
= \skalp{\nabla B(z)^T\omega^*_k, \hat z''(\omega;\eta,\xi)-v_k}\\
&\leq  \oo(\norm{\hat z''(\omega;\eta,\xi)- \hat z''(\omega;\eta,\xi_k)}),
 \end{align*}
where the first equality is due to \eqref{zv} and \eqref{vk}, and the inequality
comes from $\nabla B(z)^T\omega_k^*\in \widehat N_{T^2_\Gamma(z;w)}(v_k)$ by \eqref{regular-gamma}. This together with
\eqref{z-z} and the boundedness of $\norm{\hat w_z}$ by Lemma \ref{LemPsiTheta} yields
$ \skalp{\omega_k^*,\xi-\xi_k} \leq  \oo(\norm{\xi-\xi_k})$.
 So $\omega_k^*\in \widehat N_{T^2_\Omega(\omega;\eta)}(\xi_k)$. Therefore
 \begin{align*}\hat\sigma_{T^2_\Omega(\omega;\eta)}(\omega^*)&\leq \lim_{k\to\infty}\skalp{\omega_k^*,\xi_k} = \lim_{k\to\infty} \skalp{\omega_k^*,\nabla B(z)v_k+\nabla^2 B(z)(w,w)}\\
 &=\lim_{k\to\infty} \skalp{\nabla B(z)^T\omega_k^*, v_k}
 +\skalp{\omega^*, \nabla^2 B(z)(w,w)}
 =\lim_{k\to\infty}\skalp{z_k^*+w_k^*,v_k}+\skalp{\omega^*,\nabla^2 B(z)(w,w)}\\
 &= \lim_{k\to\infty}\skalp{z_k^*,v_k}+\skalp{\omega^*,\nabla^2 B(z)(w,w)}
 =  \hat\sigma_{T^2_\Gamma(z;w)}(\nabla B(z)^T\omega^*)+\skalp{\omega^*,\nabla^2 B(z)(w,w)},
 \end{align*}
 where the third equality comes from \eqref{regular-gamma} and the forth equality is due to \eqref{w-v}. Hence \eqref{EqLowerSuppOmega} is verified. The proof is complete.
\end{proof}

By using the formula for the second-order subderivative, we can further obtain the expression of
the proximal normal cone.

\begin{corollary}\label{cor-dom-normal}
Under the assumptions of Theorem \ref{ThTanConeAuxProbl}, for every $\omega\in\Omega\cap \O$ and $\eta\in T_\Omega(\omega)$ we have
\begin{align*} 
&\Np_\Omega(\omega;\eta)=\{\omega^*\,|\, \nabla B(z)^T\omega^*\in\Np_\Gamma(z;w)\},\ \widehat N^p_\Omega(\omega;\eta)=\{\omega^* \,|\, \nabla B(z)^T\omega^*\in \widehat N^p_\Gamma(z;w)\},\\ 
  &\dom \sigma_{T^2_\Omega(\omega;\eta)}=\{\omega^*\,|\, \nabla B(z)^T\omega^*\in \dom \sigma_{T^2_\Gamma(z;w)}\},\ \dom \hat\sigma_{T^2_\Omega(\omega;\eta)}=\{\omega^*\,|\, \nabla B(z)^T\omega^*\in \dom \hat\sigma_{T^2_\Gamma(z;w)}\},
\end{align*}
where $z:=\hat z(\omega)$ and $w:=\hat z'(\omega;\eta)$.
\end{corollary}
\begin{proof}
  For every $\omega\in\Omega\cap \O$ and $\eta\in T_\Omega(\omega)$, we know $z=\hat z(\omega)\in\Gamma$ and $w=\hat z'(\omega;\eta)\in T_\Gamma(z)$ by Proposition \ref{PropHomeo} and Theorem \ref{ThTanConeAuxProbl} respectively. According to \cite[Proposition 2.18]{BeGfrYeZhangZhou} and \eqref{EqSecOrdSubDerOmega} we have
  \begin{align*}
  \omega^*\in \Np_\Omega(\omega;\eta) & \Longleftrightarrow d^2 \delta_\Omega(\omega;\omega^*)(\eta)>-\infty \Longleftrightarrow
  d^2 \delta_\Gamma(z;\nabla B(z)^T\omega^*)(w)>-\infty \nonumber\\
  & \Longleftrightarrow \nabla B(z)^T\omega^*\in \Np_\Gamma(z;w). 
  \end{align*}
  Since $w=\hat z'(\omega;\eta)=\hat w_{z}(\eta,0)$ by Theorem \ref{ThTanConeAuxProbl},  the definition of $\hat w_z$ gives us $\nabla B(z)w=\eta$.
  From the above we thus conclude
  \begin{align*}
  \widehat N^p_\Omega(\omega;\eta)&=\Np_\Omega(\omega;\eta)\cap \{\eta\}^\perp=
  \{\omega^*\,|\, \nabla B(z)^T\omega^*\in \Np_\Gamma(z;w), \ \langle \omega^*, \eta\rangle=0 \}\\
  &= \{\omega^*\,|\, \nabla B(z)^T\omega^*\in \Np_\Gamma(z;w), \ \langle \omega^*, \nabla B(z)w \rangle=0 \}\\
  &= \{\omega^*\,|\, \nabla B(z)^T\omega^*\in \Np_\Gamma(z;w), \ \langle \nabla B(z)^T\omega^*, w\rangle=0 \}\\
  &= \{\omega^*\,|\, \nabla B(z)^T\omega^*\in \Np_\Gamma(z;w)\cap\{w\}^\perp \}\\
   &= \{\omega^*\,|\, \nabla B(z)^T\omega^*\in \widehat N^p_\Gamma(z;w)\}.
  \end{align*}
  This proves the first formula while the second follows directly from \eqref{EqSuppOmega} and \eqref{EqLowerSuppOmega}.
  \end{proof}

\if{
\begin{remark}
  Note that the above stated results are valid without any constraint qualification conditions on the system $G(z)\in D$.
\end{remark}
}\fi

\section{Variational analysis of the constraint structure \eqref{EqOmega-1}}

In this section, we apply the results of the preceding section to the set $\Omega$
 given by \eqref{EqOmega-1}, i.e., to the particular setting
\[
z:=(x,\eta), \ \ \ B(x, \eta):=\myvec{x\\ b(x)^T\eta},\ \ \ G(x,\eta):= \myvec{g(x)\\
\eta },\ \ \ D:=\gph N_P,
\]
where $g:\mathbb{R}^n\rightarrow \mathbb{R}^l$, $b:\mathbb{R}^n\rightarrow \mathbb{R}^{l\times m}$, and $P \subseteq \mathbb{R}^l$ is a convex polyhedral set. In this case,
\[
\Omega=\{B(z)\,|\, G(z)\in D\}=\{ (x, b(x)^T\eta) \,|\, \eta\in N_P(g(x))\}
\]
and
\[\Gamma=\{z\, |\, G(z)\in D\}=\{(x,\eta)\, |\, \eta\in N_P(g(x))\}.\]
We perform our analysis under the following assumption.
\begin{assumption} Let $\bar x\in g^{-1}(P)$. Assume that the following conditions hold at $\bar x$:
\begin{subequations}\label{EqBasicAss}
\begin{align}
  \label{EqBasicAss_g}\nabla g(\xb)^Td^*=0,&\ d^* \in \Spanb{N_P(g(\xb))}\ \Rightarrow d^*=0,\\
  \label{EqBasicAss_b}b(\xb)^Td^*=0,&\ d^* \in \Spanb{N_P(g(\xb))}\ \Rightarrow d^*=0.
\end{align}
\end{subequations}
\end{assumption}

\begin{proposition}\label{GP-condition-1}
  Let $\zb:=(\bar x,\dba)\in \Gamma$. Then condition \eqref{EqBasicAss_b} implies Assumption \ref{AssA1} with $\bar\omega:=(\bar x, b(\bar x)^T\dba)$ and condition \eqref{EqBasicAss_g} implies   Assumption \ref{AssA2} for every $\bar w$ satisfying $\nabla G(\zb)\wb\in T_D(G(\zb))$. Moreover, condition \eqref{EqDirNonDegen1} is fulfilled as well.
\end{proposition}
\begin{proof} We first prove that $(\bar x, \dba)\in \Sigma_{\bar\omega}$.
Consider a sequence $\omega_k\setto{\Omega}\bar\omega=(\bar x, b(\bar x)^T\bar d^*)$. Since $\omega_k\in \Omega$, then there exists $(x_k,d^*_k)$ such that
$\omega_k=(x_k,b(x_k)^Td^*_k)$ and $d^*_k\in N_P(g(x_k))$, i.e., $(x_k, d^*_k)\in B^{-1}(\omega_k)\cap \Gamma$.
Since $\omega_k\to \bar\omega$, then $x_k\to \bar x$ readily follows.

Now it needs to be shown that $d^*_k\to \dba$. First we show by contradiction that the sequence $d_k^*$ is bounded.
Assuming that the sequence $\{d_k^*\}$ is unbounded, by possibly passing to a subsequence we can assume that $\lim_{k\to\infty}\norm{d_k^*}=\infty$ and $\tilde d_k^*:=d_k^*/\norm{d_k^*}$ converges to some $\tilde d^*$ with $\norm{\tilde d^*}=1$. Since
$d^*_k\in N_P(g(x_k))$, then $\tilde d^*_k\in N_P(g(x_k))$ and $\tilde d^*\in N_P(g(\bar x))\subseteq \big(N_P(g(\bar x))\big)^+$, where we use the outer semicontinuity of the normal cone by \cite[Proposition 6.6]{RoWe98}.
Because $\omega_k=(x_k,b(x_k)^Td^*_k)$ is convergent, then $b(x_k)^Td^*_k$ is bounded, and hence
\[
b(\bar x)^T\tilde d^*=\lim_{k\to \infty}b(x_k)^T \tilde d^*_k =\lim_{k\to \infty} \frac{b(x_k)^T d^*_k}{\norm{d^*_k}}=0,
\]
contradicting \eqref{EqBasicAss_b}. Hence $d_k^*$ is bounded. Assuming that $d_k^*$ has a limit point $\tilde d^*$, then $\tilde d^*\in N_P(g(\bar x))$. Further, from $\omega_k\to\bar\omega$ we derive $b(\xb)^T\tilde d^*=b(\xb)^T\dba$. Hence $b(\xb)^T(\tilde d^*-\dba)=0$ and $\tilde d^*-\dba\in \big( N_P(g(\bar x))\big)^+$. This implies $\tilde d^*=\dba$ by \eqref{EqBasicAss_b}.
Thus $d_k^*\to \dba$ and $(x_k, d^*_k)\to (\bar x, \bar d^*)$, which proves
$(\bar x, \dba)\in \Sigma_{\bar\omega}$.

Now let us verify \eqref{EqRegB} which reads in our context as
\[\myvec{u\\b(\bar x)^Te^*+(\nabla b(\bar x)u)^T\dba}=0, \myvec{\nabla g(\bar x)u\\e^*}\in \Spanb{T_{\gph N_P}(g(\bar x),\dba)}\ \Rightarrow\ \myvec{u\\e^*}=0,\]
or equivalently rewritten as
\begin{equation}\label{EqAuxRegB}b(\bar x)^Te^*=0, \ \ \myvec{0\\e^*}\in \Spanb{T_{\gph N_P}(g(\bar x),\dba)}\ \Rightarrow\ e^*=0.\end{equation}
Using \eqref{equal-a-1} and the fact that $\K_P{(g(\bar x),\bar d^*)}$ is a convex cone, we have
\begin{align}
T_{\gph N_P}(g(\bar x),\dba)&=\gph N_{\K_P{(g(\bar x),\bar d^*)}} =\{(u,v)\,|\, u\in \K_P{(g(\bar x),\bar d^*)}, v\in N_{\K_P{(g(\bar x),\bar d^*)}}(u)\} \nonumber\\
&= \{(u,v)\,|\, u\in \K_P{(g(\bar x),\bar d^*)}, v\in \K_P{(g(\bar x),\bar d^*)}^\circ\cap [u]^\perp\} \nonumber\\
& \subseteq \K_P{(g(\bar x),\bar d^*)}\times \K_P{(g(\bar x),\bar d^*)}^\circ.  \label{tangent-perp}
\end{align}
Further since $\K_P{(g(\bar x),\bar d^*)}$ is also polyhedral, then
\begin{align}
\big(\K_P{(g(\bar x),\bar d^*)}^\circ\big)^+&=\K_P{(g(\bar x),\bar d^*)}^\circ-\K_P{(g(\bar x),\bar d^*)}^\circ=\big(T_P(g(\bar x))^\circ+[\bar d^*]\big)-\big(T_P(g(\bar x))^\circ+[\bar d^*]\big) \nonumber\\
&= \big(N_P(g(\bar x)) \big)^++[\bar d^*]=\big(N_P(g(\bar x)) \big)^+, \label{kp-perp}
\end{align}
where the last step is due to $\bar d^*\in N_P(g(\bar x))$ by assumption. Combining \eqref{tangent-perp} and \eqref{kp-perp} together yields
\[
\big(T_{\gph N_P}(g(\bar x),\dba)\big)^+\subseteq \K_P{(g(\bar x),\bar d^*)}^+\times \big(\K_P{(g(\bar x),\bar d^*)}^\circ\big)^+
\subseteq \K_P{(g(\bar x),\bar d^*)}^+\times \big(N_P(g(\bar x)) \big)^+.
\]
Hence \eqref{EqBasicAss_b} implies \eqref{EqAuxRegB} and the first assertion is verified.

In order to show the second assertion, we will prove that condition \eqref{EqDirNonDegen1} is fulfilled which implies in turn Assumption \ref{AssA2} by Lemma \ref{ThDirNonDegen}.
Denoting $\wb:=(u,e^*)$ then condition $\nabla G(\zb)\wb\in T_D(G(\zb))$ reads as
\[\myvec{\nabla g(\bar x)u\\ e^*}\in T_{\gph N_P}(g(\bar x),\dba)\]
and, denoting $p^*:=(d^*,e)$, \eqref{EqDirNonDegen1} amounts to
\begin{equation}\label{EqAuxDirNondegen}
\myvec{\nabla g(\bar x)^Td^*\\ e}=0,\ \myvec{d^*\\e}\in \Spanb{N_{\gph N_P}\big((g(\bar x),\dba);(\nabla g(\bar x)u,e^*)\big)}\ \Rightarrow\ \myvec{d^*\\ e}=0.
\end{equation}
The requirement $e=0$ is obviously fulfilled, and it remains to show $d^*=0$.
According to Theorem \ref{ThNormalConePoly},
\begin{align}
&N_{\gph N_P}\big((g(\bar x),\dba);(\nabla g(\bar x)u,e^*)\big) \label{Normal-Faces}\\
&=\bigcup\limits_{}\big\{(\F_1-\F_2)^\circ\times (\F_1-\F_2)\,|\, \F_1, \F_2\in \F(\K_P(g(\bar x),\dba)), \ \ \nabla g(\bar x)u\in \F_2\subseteq\F_1\subseteq [e^*]^\perp\big\}. \nonumber
\end{align}
Take $\F_1, \F_2\in \F(\K_P(g(\bar x),\dba))$ with $\F_2\subseteq\F_1$. Since $\F_1,\F_2$ are faces of the critical cone $\K_P(g(\bar x),\bar d^*)$, we have $\lin{\K_P(g(\bar x),\bar d^*)}\subseteq \F_2\subseteq \F_1\subseteq \K_P(g(\bar x),\bar d^*)$, implying
\begin{equation}\label{F1-F2-1}
\lin{\KPg}\subseteq\F_1-\F_2\subseteq \K_P(g(\bar x),\bar d^*)-\K_P(g(\bar x),\bar d^*)=\Spanb{\K_P(g(\bar x),\bar d^*)},
\end{equation}
and consequently
\begin{align}
(\F_1-\F_2)^\circ & \subseteq \big(\lin{\K_P(g(\bar x),\bar d^*)}\big)^\circ =
\big(\K_P(g(\bar x),\bar d^*)\cap \big(-\K_P(g(\bar x),\bar d^*)\big) \big)^\circ \nonumber\\
& =\K_P(g(\bar x),\bar d^*)^\circ+\big(-\K_P(g(\bar x),\bar d^*)\big)^\circ
 =\K_P(g(\bar x),\bar d^*)^\circ-\K_P(g(\bar x),\bar d^*)^\circ \nonumber\\
 &= \big(\K_P(g(\bar x),\bar d^*)^\circ\big)^+=\big(N_P(g(\bar x)) \big)^+, \label{F1-F2-2}
 \end{align}
 where the second equality comes from the fact that $\K_P(g(\bar x),\bar d^*)$ is a convex polyhedral cone and
 the last equality follows from \eqref{kp-perp}. Combining \eqref{Normal-Faces}, \eqref{F1-F2-1} and \eqref{F1-F2-2} together yields
 \begin{equation}\label{normal-space}
 \Spanb{N_{\gph N_P}\big((g(\bar x),\dba);(\nabla g(\bar x)u,e^*)\big)}\subseteq \big(N_P(g(\bar x)) \big)^+\times \Spanb{\K_P(g(\bar x),\bar d^*)}.
 \end{equation}
 Therefore \eqref{EqAuxDirNondegen} is implied by \eqref{EqBasicAss_g}.
 \end{proof}

Combining all the material provided in the previous sections we obtain the following main result.
\begin{theorem}\label{Thm 4.4}
Let $\bar x\in g^{-1}(P)$ and $\dba\in N_P(g(\bar x))$ be given and assume that conditions \eqref{EqBasicAss} are fulfilled.
For $\bar\omega:=(\bar x,b(\bar x)^T\dba)$, the following statements hold true.
\begin{enumerate}
 \item[\rm (i)]  $T_\Omega(\bar\omega)=\Big\{\myvec{u\\ (\nabla b(\bar x)u)^T\dba+ b(\bar x)^Te^*} \, \left|\, \myvec{\nabla g(\bar x)u\\e^*}\in T_{{\rm gph}N_P}(g(\bar x),\bar d^*)\right.\Big\}$.
 \item[\rm (ii)] Let $\eta\in T_\Omega(\bar\omega)$ and let $(u,e^*)$ denote the element fulfilling
 \begin{equation}\label{ue}
 \eta=(u,(\nabla b(\bar x)u)^T\dba+ b(\bar x)^Te^*)  \ \ {\rm and} \ \   (\nabla g(\bar x)u,e^*)\in T_{{\rm gph}N_P}(g(\bar x),\bar d^*).
 \end{equation}
 Then
 \begin{align}\label{EqDirLimNormalOmega}\lefteqn{N_\Omega(\bar\omega;\eta)}\\
 \nonumber&=\Big\{\myvec{-\nabla (b^T\dba)(\bar x)^T r+\nabla g(\bar x)^T f^*\\r}
 \,\left|\, \myvec{f^*\\ b(\bar x)r}\in N_{\gph N_P}\big((g(\bar x),\dba);(\nabla g(\bar x)u,e^*)\big)\right.\Big\},\\
 \label{EqSecOrdTanOmega}&T^2_\Omega(\bar \omega;\eta)\\
 &=\Big\{\left.\myvec{\zeta \\ 2(\nabla b(\bar x)u)^Te^*+\big( \nabla b(\bar x)\zeta+ \nabla^2 b(\bar x)(u,u)\big)^T\dba+b(\bar x)^T\xi^*} \, \right|\, \nonumber \\
 \nonumber&\qquad\qquad\qquad\qquad\myvec{\nabla g(\bar x)\zeta+\nabla^2 g(\bar x)(u,u)\\ \xi^*}\in T^2_{\gph N_P}\big((g(\bar x),\dba);(\nabla g(\xb)u,e^*)\big)\Big\}.
 \end{align}
 \item[\rm (iii)]  Let $\eta\in T_\Omega(\bar \omega)$ and let $(u,e^*)$ satisfy \eqref{ue}. Then
 \begin{align}
 \label{EqDomSigmaOmega1}&\lefteqn{\dom\sigma_{T^2_\Omega(\bar\omega;\eta)}=\widehat N^p_\Omega(\bar\omega;\eta)}\\
 \nonumber&=\Big\{\myvec{-\nabla (b^T\dba)(\bar x)^Tr+\nabla g(\bar x)^Tf^*\\r} \, \left|\,
 \myvec{f^*\\ b(\bar x)r}\in \widehat N_{T_{{\rm gph}N_P}(g(\bar x),\dba)}(\nabla g(\bar x)u, e^*)\right. \Big\},\\
 \label{EqDomHatSigmaOmega1}&\dom\hat\sigma_{T^2_\Omega(\bar\omega;\eta)}=N_\Omega(\bar\omega;\eta).
 \end{align}
 \item[\rm (iv)] Let $\eta\in T_\Omega(\bar \omega)$ and let $(u,e^*)$ satisfy \eqref{ue}. For every $\omega^*\in \dom\sigma_{T^2_\Omega(\bar\omega;\eta)}$ (respectively $\omega^*\in\dom\hat\sigma_{T^2_\Omega(\bar\omega;\eta)}$), let $(r,f^*)$ denote the elements fulfilling
 \begin{equation}\label{w-star}
   \omega^*=\myvec{-\nabla (b^T\dba)(\bar x)^Tr+\nabla g(\bar x)^Tf^*\\r}
   \end{equation}
   and
   \begin{equation}\label{f-r-region}
   \myvec{f^*\\ b(\bar x)r}\in \widehat
   N_{T_{{\rm gph}N_P}(g(\bar x),\dba)}(\nabla g(\bar x)u, e^*)
      \ \big(\mbox{respectively }N_{\gph N_P}\big((g(\bar x),\dba);(\nabla g(\bar x)u,e^*)\big)\big).
 \end{equation}
 Then
 \[ \sigma_{T^2_\Omega(\bar\omega;\eta)}(\omega^*)(\mbox{resp. }\hat\sigma_{T^2_\Omega(\bar\omega;\eta)}(\omega^*))=\skalp{r, 2(\nabla b(\bar x)u)^Te^*+\nabla^2 b(\bar x)(u,u)^T\dba}-\skalp{f^*,\nabla^2g(\bar x)(u,u)},
 \]
 ${\rm d^2\,}\delta_\Omega(\bar \omega;\omega^*)(\eta)=-\sigma_{T^2_\Omega(\bar\omega;\eta)}(\omega^*)$.
\end{enumerate}
\end{theorem}

\begin{proof} (i) Let $\zb:=(\bar x, \dba)$. Since $B(\bar z)=\bar \omega$, we get
\begin{equation}\label{z-hat}
 \hat z(\bar \omega)=\bar z
 \end{equation}
under the condition \eqref{EqBasicAss}. According to Lemma \ref{ThDirNonDegen} and \eqref{EqTanAuxProbl}, we know that
\begin{equation}\label{tangent-two-part}
T_\Omega(\bar \omega)=\{\nabla B(\zb)w \, |\, w\in T_\Gamma(\bar z)\}=\{\nabla B(\zb)w \, |\, \nabla G(\bar z)w\in T_D(G(\bar z))\}.
\end{equation}
For $w:=(u,e^*)\in T_\Gamma(\zb)$, we have the characterization
\begin{equation}\label{G-part}
\nabla G(\zb)w=(\nabla g(\bar x)u,e^*)\in T_D(G(\zb))=T_{{\rm gph}N_P}(g(\bar x),\bar d^*).
\end{equation}
Since $\nabla B(\zb)w=(u, (\nabla b(\bar x)u)^T\dba+b(\bar x)^Te^*)$, the formula for $T_\Omega(\bar\omega)$ follows from
\eqref{tangent-two-part} and \eqref{G-part}.

(ii) Let $\eta\in T_\Omega(\bar \omega)$. The existence of $(u,e^*)$ satisfying \eqref{ue} is ensured by the formula for $T_\Omega(\bar \omega)$ established in (i). Since $\hat z'(\bar\omega;\eta)=\hat w_{\bar z}(\eta,0)$ by Theorem \ref{ThTanConeAuxProbl}, then
\[\nabla B(\bar z)\hat z'(\bar\omega;\eta)=\eta  \ \ \ {\rm and } \ \ \ \nabla G(\bar z)\hat z'(\bar\omega;\eta)\in \big(T_D(G(\bar z))\big)^+
=\big(T_{{\rm gph}N_P}(g(\bar x),\bar d^*)\big)^+,
\]
 by the definition of $\hat w_{\bar z}$ given in Lemma \ref{LemPsiTheta}.
 Let $\hat z'(\bar\omega;\eta):=(a,a^*)$. Then
 $\eta=\nabla B(\bar z)\hat z'(\bar\omega;\eta)=(a,(\nabla b(\bar x)a)^T\dba+ b(\bar x)^T a^*)$ and
 $\nabla G(\bar z)\hat z'(\bar\omega;\eta)=(\nabla g(\bar x)a,a^*)$. Comparing this with \eqref{ue}, we obtain
 $a=u$, $b(\bar x)^T(a^*-e^*)=0$ and $(0,a^*-e^*)\in \big(T_{{\rm gph}N_P}(g(\bar x),\bar d^*)\big)^+$, which further implies
 $a^*=e^*$ by \eqref{EqAuxRegB}. Hence
\begin{equation}\label{z-u-e}
\hat z'(\bar\omega;\eta)=(u,e^*)=:w.
\end{equation}

 Note that $N_\Omega(\bar\omega;\eta)=\{\omega^*\,|\, \nabla B(\bar z)^T\omega^*\in N_\Gamma(\bar z;w)\}$ by \eqref{EqLimNormalAuxProbl} and $N_\Gamma(\bar z;w)=\nabla G(\bar z)^TN_D(G(\bar z);\nabla G(\bar z)w)$ by Lemma \ref{ThDirNonDegen}. Hence
 \begin{equation}\label{normal-two-part}
 \omega^*\in N_\Omega(\bar\omega;\eta)\ \Longleftrightarrow\
 \nabla B(\bar z)^T\omega^*=\nabla G(\bar z)^Tq^* \  {\rm for \ some} \ q^*\in N_D(G(\bar z);\nabla G(\bar z)w).
 \end{equation}
 Denote $\omega^*:=(r^*,r)$ and $q^*:=(f^*,f)$. Since
 \[\nabla B(\zb)^T\omega^*=\big(r^*+\nabla(b^T\dba)(\bar x)^Tr, b(\bar x)r\big), \ \ \nabla G(\zb)^Tq^*=(\nabla g(\bar x)^Tf^*, f), \ \ \nabla G(\bar z)w=(\nabla g(\bar x)u,e^*),\]
 then \eqref{EqDirLimNormalOmega} follows from \eqref{normal-two-part}.

 According to \cite[Corollary 1]{BeGfrOut19}, since $T^2_D(G(\bar z);\nabla G(\bar z)w)$ is nonempty, we know that $T^2_\Gamma(\bar z;w)$ is also nonempty under Assumption \ref{AssA2}, which is ensured by condition \eqref{EqBasicAss} via
 Proposition \ref{GP-condition-1}. Hence $T^2_\Omega(\bar \omega;\eta)$ is nonempty by \eqref{EqSecOrdTanSetAuxProbl}.
Further combining with \eqref{EqSecOrdTanGamma}, \eqref{z-hat} and \eqref{z-u-e} yields
  \begin{equation}\label{second-tangent-twopart}
  T^2_\Omega(\bar \omega;\eta)=\big\{\nabla B(\bar z)v+\nabla^2B(\bar z)(w,w)\,|\, \nabla G(\bar z)v+\nabla^2G(\bar z)(w,w)\in T^2_D(G(\bar z);\nabla G(\bar z)w)\big\}.
  \end{equation}
  Denote $v:=(\zeta,\xi^*)$. Taking into account
  \[
\nabla B(\zb)v=\myvec{\zeta \\ \big(\nabla b(\bar x)\zeta\big)^T\dba+b(\bar x)^T\xi^*},\ \ \ \nabla G(\zb)v=\myvec{\nabla g(\bar x)\zeta\\\xi^*},
\]
and
  \begin{align}\label{second-derivative}
  \nabla^2B(\zb)(w,w)=\myvec{0\\2(\nabla b(\bar x)u)^Te^*+\nabla ^2 b(\bar x)(u,u)^T\dba},\ \ \nabla^2G(\zb)(w,w)=\myvec{\nabla^2g(\bar x)(u,u)\\0},
  \end{align}
   \eqref{EqSecOrdTanOmega} follows from \eqref{second-tangent-twopart}.

(iii)  Combining Corollary \ref{cor-dom-normal}, \eqref{z-hat}, \eqref{z-u-e} with Corollary \ref{CorDirNonDegen} yields
\[
\dom \sigma_{T^2_\Omega(\bar\omega;\eta)}=\left\{\omega^*\,|\, \nabla B(\zb)^T\omega^*\in \nabla G(\zb)^T\widehat N_{T_{\gph N_P}(G(\zb))}(\nabla G(\zb)w)\right\}.
\]
Since $\dom \sigma_{T^2_\Gamma(\zb; w)}=\widehat N_\Gamma^p(\zb, w)$ by Corollary \ref{CorDirNonDegen}, then
$\dom\sigma_{T^2_\Omega(\bar\omega;\eta)}=\widehat N^p_\Omega(\bar\omega;\eta)$ by Corollary \ref{cor-dom-normal}.
The formula \eqref{EqDomSigmaOmega1} can be obtained by  restricting that the range of $q^*$ in \eqref{normal-two-part} belongs to the set $\widehat
   N_{T_{{\rm gph}N_P}(g(\bar x),\dba)}(\nabla g(\bar x)u, e^*)$ and then following the argument as in the proof of \eqref{EqDirLimNormalOmega}.

It follows from \eqref{EqLimNormalAuxProbl} and Corollaries \ref{CorDirNonDegen} and \ref{cor-dom-normal} that
\[
\dom\hat\sigma_{T^2_\Omega(\bar\omega;\eta)}=\{\omega^*\, |\, \nabla B(\bar z)^T\omega^*\in \dom \hat\sigma_{T^2_\Gamma(z:w)}\}
= \{\omega^*\, |\, \nabla B(\bar z)^T\omega^*\in N_\Gamma(\bar z; w)\}=N_\Omega(\bar \omega;\eta),
\]
and hence \eqref{EqDomHatSigmaOmega1} holds.

(iv)  Using \eqref{z-hat} and \eqref{z-u-e}, according to \eqref{EqSuppOmega} and Corollary \ref{CorDirNonDegen}(ii), we know
\begin{align}
\sigma_{T^2_\Omega(\bar\omega;\eta)}(\omega^*)&=\sigma_{T^2_\Gamma(\bar z;w)}(\nabla B(\bar z)^T\omega^*)+\skalp{\omega^*,\nabla^2 B(\bar z)(w,w)} \nonumber \\
&= -\langle p^*_0, \nabla^2 G(\bar z)(w,w)\rangle+\skalp{\omega^*,\nabla^2 B(\bar z)(w,w)},\label{p-1}
\end{align}
where $p^*_0\in \Lambda^s_{\nabla B(\bar z)^T\omega^*}(\bar z; w)$, i.e., $p^*_0\in \widehat N_{T_D(G(\bar z))}(\nabla G(\bar z)w)$ and
\begin{equation}\label{B-G-1}
\nabla B(\bar z)^T\omega^*=\nabla G(\bar z)^Tp^*_0.
\end{equation}
Denoting $r^*:=-\nabla (b^T\dba)(\bar x)^Tr+\nabla g(\bar x)^Tf^*$, then
$\omega^*=(r^*,r)$ by \eqref{w-star}. So
\begin{equation}\label{B-G-2}
\nabla B(\bar z)^T\omega^*=\myvec{\nabla g(\bar x)^T f^*\\ b(\bar x)r}=\nabla G(\bar z)^T\myvec{f^*\\ b(\bar x)r}.
\end{equation}
Comparing \eqref{B-G-1} with \eqref{B-G-2} yields $p^*_0=(f^*, b(\bar x)r)$, since the set $\Lambda^s_{\nabla B(\bar z)^T\omega^*}(\bar z; w)$ is a singleton by Corollary \ref{CorDirNonDegen}(ii). Substituting $p^*_0$ by $(f^*, b(\bar x)r)$ into \eqref{p-1} and using \eqref{second-derivative} yields the formula
for $\sigma_{T^2_\Omega(\bar\omega;\eta)}(\omega^*)$.  The formula for $\hat\sigma_{T^2_\Omega(\bar\omega;\eta)}(\omega^*)$
follows from \eqref{EqLowerSuppOmega} and Corollary \ref{CorDirNonDegen}(i) by restricting $p^*_0$ in $N_D(G(\bar z);\nabla G(\bar z)w)= N_{\gph N_P}\big((g(\bar x),\dba);(\nabla g(\bar x)u,e^*)\big)$ as required in \eqref{f-r-region}.
 Finally the formula for ${\rm d^2\,}\delta_\Omega(\bar \omega;\omega^*)(\eta)$ comes from \eqref{EqSecOrdSubDerOmega},
 \eqref{EqSuppOmega} and Corollary \ref{CorDirNonDegen}(ii).
\end{proof}

\section{Second-order optimality conditions for GEPs}

Since the optimization problem with generalized equation constraints (GEP) can be rewritten as the constrained problem GP,
we can employ the expressions established in Section 4 to derive the second-order optimality conditions for GEP.

\begin{theorem}[Second-order optimality conditions for GEPs]\label{second-order-appl-1}
Let  $\bar{x}$ be a feasible point of problem GEP. Assume that
 the condition \eqref{EqBasicAss} holds at $\bar x$. Then there exists a unique $\eta^* \in N_P(g(\bar x))$ such that $0=F(\bar x) +b(\bar x)^T\eta^*$.
Let a direction $d$ satisfy
\begin{equation}\label{second-1}
\nabla f(\bar x)d \leq 0,\ \  0=\nabla F(\bar x)d+(\nabla b(\bar x)d)^T\eta^*  +b(\bar x)^T e^*,
\end{equation}
where $(\nabla g(\bar x)d, e^*)\in T_{{\rm gph}N_P}(g(\bar x),\eta^*)$.
\begin{itemize}
\item[{\rm (i)}] Let $\bar x$ be a local optimal solution of GEP. Suppose that
\begin{equation}\label{nnamcq-2}
\left. \begin{array}{c}
\nabla F(\bar x)^T\lambda+\nabla \big(b^T\eta^*\big)(\bar x)^T\lambda-\nabla g(\bar x)^T\tau^*=0, \\
 (\tau^*,b(\bar x)\lambda)\in N_{{\rm gph}N_P}((g(\bar x),\eta^*);(\nabla g(\bar x)d,e^*))
\end{array}
 \right\} \Longrightarrow \lambda=0.
\end{equation}
Then there exist multipliers  $\lambda, \tau^*$ such that
   the first-order condition
   \begin{equation}\label{second-3}
     \nabla f(\bar x)-\nabla F(\bar x)^T\lambda-\nabla \big(b^T\eta^*\big)(\bar x)^T\lambda+\nabla g(\bar x)^T\tau^*=0,\end{equation}
     and the second-order condition
    \begin{align}\label{second-4}
   &\nabla^2 f(\bar x)(d,d) -\langle \lambda, (\nabla^2F(\bar x)(d,d))+\nabla^2 b(\bar x)(d,d)^T\eta^*+ 2(\nabla b(\bar x)d)^Te^*\rangle \nonumber\\
   &+\skalp{\tau^*,\nabla^2g(\bar x)(d,d)}\geq 0,
    \end{align}
    where $(\tau^*,b(\bar x)\lambda)\in N_{{\rm gph}N_P}((g(\bar x),\eta^*);(\nabla g(\bar x)d,e^*))$.
\item[\rm (ii)] Suppose that for every nonzero $d$ satisfying \eqref{second-1}, there is $(\alpha,\lambda,\tau^*)$ not all  equal to zero satisfying $\alpha\geq 0$ and
    \[
\left\{
\begin{array}{l}
   \alpha\nabla f(\bar x)-\nabla F(\bar x)^T\lambda-\nabla \big(b^T\eta^*\big)(\bar x)^T\lambda+\nabla g(\bar x)^T\tau^*=0, \\
(\tau^*,b(\bar x)\lambda)\in  \widehat N_{T_{{\rm gph}N_P}(g(\bar x), \eta^*)}(\nabla g(\bar x)d, e^*)
\end{array}\right.
\]
   and
    \begin{align}\label{second-5}
   & \alpha\nabla^2 f(\bar x)(d,d) -\langle \lambda, (\nabla^2F(\bar x)(d,d))+\nabla^2 b(\bar x)(d,d)^T\eta^*+ 2(\nabla b(\bar x)d)^Te^*\rangle \nonumber\\
   &+\skalp{\tau^*,\nabla^2g(\bar x)(d,d)}> 0,
    \end{align}
    then $\bar x$ is an essential local minimizer of second order for GEP.
\end{itemize}
\end{theorem}

\begin{proof}
(i).
Since $\bar x$ is a feasible point,
$-F(\bar x)=b(\bar x)^T\eta^*$ for $\eta^*\in N_P(g(\bar x))$, where $\eta^*$ is unique under the condition \eqref{EqBasicAss_b}.
Let $\varphi(x):=(x,-F(x))$. Denote $\bar \omega:=(\bar x, b(\bar x)^T\eta^*)$. Then $\bar \omega=(\bar x,-F(\bar x))=\varphi(\bar x)\in \Omega$. Note that
\begin{align*}
C(\bar x)&=\{d\,|\, \nabla f(\bar x)d\leq 0, \nabla \varphi(\bar x)d\in T_\Omega(\varphi(\bar x))\}
=\{d| \nabla f(\bar x)d\leq 0, (d,-\nabla F(\bar x)d)\in T_\Omega(\bar \omega)\}\\
&=\{d\, |\, \nabla f(\bar x)d\leq 0, -\nabla F(\bar x)d=(\nabla b(\bar x)d)^T\eta^*+b(\bar x)^Te^*, (\nabla g(\bar x)d, e^*)\in T_{{\rm gph}N_P}(g(\bar x),\eta^*)\},
\end{align*}
where the third equation follows from Theorem \ref{Thm 4.4}(i). Hence the direction $d$ satisfying \eqref{second-1} belongs to the critical cone $C(\bar x)$. For such $d$, according to Proposition \ref{Prop2.17}, we need to show that the set-valued mapping $x\tto \varphi(x)-\Omega$ is metrically subregular at $(\bar x,0)$ in direction $d$, which can be ensured by the condition
\begin{equation}\label{dir-mscq}
\nabla \varphi(\bar x)^T\myvec{\mu \\ \lambda}=0, \ \myvec{\mu \\ \lambda}\in N_\Omega(\varphi(\bar x); \nabla \varphi(\bar x)d) \Longrightarrow \lambda=0.
\end{equation}
Let $\eta:=(d,-\nabla F(\bar x)d)$. Then $\eta=\nabla \varphi(\bar x)d=(d, (\nabla b(\bar x)d)^T\eta^*  +b(\bar x)^T e^*)$ by \eqref{second-1}. Hence $\eta\in T_\Omega(\bar \omega)$ by \eqref{ue}. Note that
 $T^2_\Omega(\varphi(\bar x);\nabla \varphi(\bar x)d)=T^2_\Omega(\bar \omega; \eta)$ is nonempty according to the argument preceding \eqref{second-tangent-twopart}. It follows from \eqref{EqDirLimNormalOmega} that  \begin{equation}\label{normalcone-1}
(\mu, \lambda)\in N_\Omega(\varphi(\bar x);\nabla \varphi(\bar x)d) \Longleftrightarrow
\left\{\begin{array}{l} \mu=-\nabla \big(b^T\eta^*\big)(\bar x)^T\lambda+\nabla g(\bar x)^T\tau^*, \\
 (\tau^*,b(\bar x)\lambda)\in N_{{\rm gph}N_P}((g(\bar x),\eta^*);(\nabla g(\bar x)d,e^*)).
 \end{array}\right. \end{equation}
 Since $\nabla \varphi(\bar x)^T \myvec{\mu \\ \lambda}=\mu-\nabla F(\bar x)^T\lambda$, then \eqref{dir-mscq} takes the form \eqref{nnamcq-2}.

Let $L(x,\mu,\lambda):=f(x)+\langle \mu,x \rangle-\langle \lambda, F(x)\rangle  $ and \[\Lambda(\bar x;d):=\{\lambda\,|\,\nabla_x L(\bar x,\mu,\lambda)=0, (\mu,\lambda)\in N_\Omega(\varphi(\bar x);\nabla \varphi(\bar x)d)\}.\]
Then
\begin{align}
(\mu,\lambda)\in \Lambda(\bar x;d)&\Longleftrightarrow \left\{\begin{array}{l} \nabla f(\bar x)+\mu-\nabla F(\bar x)^T\lambda=0,\\
 \mu=-\nabla \big(b^T\eta^*\big)(\bar x)^T\lambda+\nabla g(\bar x)^T\tau^*, \\
 (\tau^*,b(\bar x)\lambda)\in N_{{\rm gph}N_P}((g(\bar x),\eta^*);(\nabla g(\bar x)d,e^*)).
 \end{array}\right. \nonumber\\
 & \Longleftrightarrow \left\{\begin{array}{l} \nabla f(\bar x)-\nabla F(\bar x)^T\lambda
 -\nabla \big(b^T\eta^*\big)(\bar x)^T\lambda+\nabla g(\bar x)^T\tau^*=0,\\
 (\tau^*,b(\bar x)\lambda)\in N_{{\rm gph}N_P}((g(\bar x),\eta^*);(\nabla g(\bar x)d,e^*)).
 \end{array}\right. \label{first-order optim}
 \end{align}
Since $\Lambda(\bar x;d)$ is nonempty under metric subregularity, then we can obtain the existence of $\lambda$ and $\tau^*$ satisfying
\eqref{second-3} by \eqref{first-order optim}.

Take $\omega^*:=(\mu, \lambda)\in N_\Omega(\varphi(\bar x);\nabla \varphi(\bar x)d)=\dom\hat\sigma_{T^2_\Omega(\bar\omega;\eta)}$, where the equality is due to \eqref{EqDomHatSigmaOmega1}. Then $\omega^*=(-\nabla \big(b^T\eta^*\big)(\bar x)^T\lambda+\nabla g(\bar x)^T\tau^*,\lambda)$ by \eqref{normalcone-1}. Hence it follows from Theorem \ref{Thm 4.4}(iv) that
\begin{align}\label{hatsigma-GP}
&\hat\sigma_{T^2_\Omega(\varphi(\bar x);\nabla \varphi(\bar x)d)}(\mu,\lambda)=\hat\sigma_{T^2_\Omega(\bar\omega;\eta)}(\omega^*)\nonumber\\
&=\skalp{\lambda, 2(\nabla b(\bar x)d)^Te^*+\nabla^2 b(\bar x)(d,d)^T\eta^*}-\skalp{\tau^*,\nabla^2g(\bar x)(d,d)}.
\end{align}
This together with the fact $\nabla^2_{xx}L(\bar x,\mu,\lambda)(d,d)=\nabla^2f(\bar x)(d,d)-\langle \lambda, \nabla^2F(\bar x)(d,d)\rangle $ and \eqref{add-second-thereom-1} yields \eqref{second-4}.

(ii). From \eqref{EqDomSigmaOmega1}, we know
\begin{equation}
(\mu, \lambda)\in \widehat  N^p_\Omega(\varphi(\bar x);\nabla \varphi(\bar x)d)= \widehat N^p_\Omega(\bar\omega;\eta) \Longleftrightarrow \left\{
\begin{array}{l}
\mu=-\nabla \big(b^T\eta^*\big)(\bar x)^T\lambda+\nabla g(\bar x)^T\tau^* \\
(\tau^*,b(\bar x)\lambda)\in  \widehat N_{T_{{\rm gph}N_P}(g(\bar x),\eta^*)}(\nabla g(\bar x)d, e^*).
\end{array}
\right. \label{mu-lambda}
\end{equation}
Let $L^\alpha(x,\mu,\lambda):=\alpha f(x)+\langle \mu,x \rangle-\langle \lambda, F(x)\rangle$. The condition required on $(\alpha,\mu,\lambda)$ in Proposition \ref{Prop2.17} is
\begin{equation}\label{alpha-1}
\alpha\geq 0,\ \  (\alpha,\mu,\lambda)\neq 0, \ \ (\mu, \lambda)\in \widehat  N^p_\Omega(\varphi(\bar x);\nabla \varphi(\bar x)d),\ \  \nabla_x L^\alpha(\bar x,\mu,\lambda)=0.
\end{equation}
We claim that the above condition is equivalent to
\begin{equation}\label{alpha-2}
\left\{
\begin{array}{l}
\alpha\geq 0, \ \ (\alpha,\lambda,\tau^*)\neq 0, \ \ \alpha\nabla f(\bar x)-\nabla F(\bar x)^T\lambda-
\nabla \big(b^T\eta^*\big)(\bar x)^T\lambda+\nabla g(\bar x)^T\tau^*=0, \\
(\tau^*,b(\bar x)\lambda)\in \widehat N_{T_{{\rm gph}N_P}(g(\bar x), \eta^*)}(\nabla g(\bar x)d, e^*).\end{array}\right.
\end{equation}
Comparing \eqref{alpha-1} and \eqref{alpha-2} on the basis of \eqref{mu-lambda}, it only remains to show that
$\alpha=0,\lambda=0,\mu\neq 0$ is equivalent to $\alpha=0,\lambda=0,\tau^*\neq 0$ in that setting.
In fact, if $\alpha=0,\lambda=0,\mu\neq 0$, then $\mu=\nabla g(\bar x)^T\tau^*$ by \eqref{mu-lambda}, and hence $\tau^*\neq 0$.
Conversely, if $\alpha=0,\lambda=0,\tau^*\neq 0$, then according to \eqref{mu-lambda} we know
\begin{equation}\label{mu-1}
\mu=\nabla g(\bar x)^T\tau^* \ \ {\rm and}\ \
(\tau^*,0)\in \widehat N_{T_{{\rm gph}N_P}(g(\bar x), \eta^*)}(\nabla g(\bar x)d, e^*).
\end{equation}
Note that
\begin{align}\widehat N_{T_{{\rm gph}N_P}(g(\bar x), \eta^*)}(\nabla g(\bar x)d, e^*) &
\subseteq N_{T_{{\rm gph}N_P}(g(\bar x), \eta^*)}(\nabla g(\bar x)d, e^*)=N_{{\rm gph}N_P}((g(\bar x), \eta^*);(\nabla g(\bar x)d, e^*)) \nonumber \\
&\subseteq \big(N_P(g(\bar x)) \big)^+\times \Spanb{\K_P(g(\bar x),\eta^*)}, \label{mu-2}
\end{align}
where the equality is due to \eqref{relation-normal} (by requiring $G:=I$ and $D:={\rm gph}N_P$ in Lemma \ref{ThDirNonDegen}) and the last step follows from \eqref{normal-space}. Putting \eqref{mu-1} and \eqref{mu-2} together yields
$\mu=\nabla g(\bar x)^T\tau^*$ with $\tau^*\in \big(N_P(g(\bar x)) \big)^+$.
Hence $\mu\neq 0$ by \eqref{EqBasicAss_g}, since $\tau^*\neq 0$.

Similar to the above argument on \eqref{hatsigma-GP}, take $\omega^*:=(\mu, \lambda)\in \widehat  N^p_\Omega(\varphi(\bar x);\nabla \varphi(\bar x)d)=\dom\sigma_{T^2_\Omega(\bar\omega;\eta)}$ by \eqref{EqDomSigmaOmega1}.
Then it follows from Theorem \ref{Thm 4.4}(iv) that
\begin{align*}
&{\rm d^2\,}\delta_\Omega(\varphi(\bar x),\omega^*)(\nabla \varphi(\bar x)d) = {\rm d^2\,}\delta_\Omega(\bar \omega;\omega^*)(\eta)=-\sigma_{T^2_\Omega(\bar\omega;\eta)}(\omega^*) \\
& =-\skalp{\lambda, 2(\nabla b(\bar x)d)^Te^*+\nabla^2 b(\bar x)(d,d)^T\eta^*}+\skalp{\tau^*,\nabla^2g(\bar x)(d,d)}.
 \end{align*}
 Hence \eqref{second-general} takes the form \eqref{second-5}. So the desired result follows from Proposition \ref{Prop2.17}.
 \end{proof}

\section{Second-order optimality conditions for MPVIs}

For illustration, consider the mathematical program with variational inequality constraints from the introduction
\begin{flalign*}
\begin{split}
\mbox{(MPVI)} \hspace{46mm} \min\limits_{x,y} & \ \  f(x,y)\\
{\rm s.t.}
& \ \  \langle F(x,y),y'-y\rangle \geq 0, \ \ \ \forall y'\in \Gamma(x)
\end{split}&
\end{flalign*}
with twice continuously differentiable functions $f$, $F$, and $\psi$ of appropriate dimensions and
where $\Gamma(x):=\{y\,|\, \psi(x,y)\in P\}$ for the particular choice $P :=\R^l_-$.
For this special case, the expressions of $T_{{\rm gph}N_P}$, $N_{{\rm gph}N_P}$, and $\widehat N_{T_{{\rm gph}N_P}}$
can be obtained by simply calculating the corresponding ones for $\R_-$. Note that ${\rm gph}N_{\R_-}=(\R_-,0)\cup(0,\R_+)$. Then
  \[
T_{{\rm gph}N_{\R_-}}(a,b)=\begin{cases}
(\R,0) & a<0, b=0, \\
(\R_{-},0) \cup (0,\R_+) & a=0, b=0,\\
(0,\R) & a=0, b>0,
\end{cases}
\]
\[
 N_{{\rm gph}N_{\R_-}}((a,b);(c,d))=
 \begin{cases}
 (0,\R) & a<0, b=0, c\in \R, d=0, \\
 (0,\R) & a=0, b=0, c<0, d=0,\\
 (0,\R)\cup (\R,0)\cup (\R_+,\R_-) & a=0,b=0,c=0,d=0,\\
 (\R,0) & a=0, b=0, c=0, d>0,\\
 (\R,0) & a=0, b>0, c=0, d\in \R,
 \end{cases}
\]
and
\[
 \widehat N_{T_{{\rm gph}N_{\R_-}}(a,b)}(c,d)=
 \begin{cases}
 (0,\R) \ \ \ \ \ \ \ \ \ \ \ \ \ & a<0, b=0, c\in \R, d=0, \\
 (0,\R) & a=0, b=0, c<0, d=0, \\
 (\R_+,\R_-) & a=0,b=0,c=0, d=0,\\
 (\R,0) & a=0, b=0, c=0, d>0,\\
 (\R,0) & a=0, b>0, c=0, d\in \R.
 \end{cases}
\]

\begin{theorem}[Second-order optimality conditions for MPVIs]\label{second-order-appl-1}
Let  $\bar z:=(\bar{x},\bar{y})$ be a feasible point of MPVI. Assume that $\psi(x,y)$ is convex in variable $y$ and
that the lower level program satisfies the nondegeneracy condition
\begin{equation}\label{nondege-condition}
\nabla_y \psi(\bar z)^T d^*=0,\ d^* \in \Spanb{N_P(\psi(\bar z))}\ \Rightarrow d^*=0.
\end{equation}
 Then there exists a unique lower level multiplier $\eta^* \in N_P(\psi(\zb))$ such that $0=F(\zb) +\nabla_y \psi(\zb)^T\eta^*$.
Let direction $d$  satisfy
\begin{equation}\label{direction-bilevel}
\nabla f(\zb)d \leq 0,\ \  0=\nabla [F +  \nabla_y \psi^T\eta^*](\zb)d +\nabla_y \psi(\zb)^T e^*,
\end{equation}
where $(\nabla \psi(\zb )d,e^*)\in T_{{\rm gph}N_P}(\psi(\zb),\eta^*)$.
\begin{itemize}
\item[{\rm (i)}] Let $\bar z$ be a local optimal solution of MPVI. Suppose that
\[
\left. \begin{array}{c}
\nabla \big(F+ \nabla_y \psi^T\eta^*\big)(\bar z)^T\lambda - \nabla \psi(\zb)^T \tau^*=0, \\
 (\tau^*,\nabla_y \psi(\zb)\lambda ) \in N_{{\rm gph} N_P}\left ((\psi(\zb),\eta^*);(\nabla \psi(\zb)d,e^*)\right )
\end{array}
 \right\} \Longrightarrow \lambda=0.
\]
   Then there exist  multipliers  $\lambda, \tau^*$ such that
   the first-order condition
\[
\nabla f(\zb) -\nabla \big(F+\nabla_y \psi^T\eta^*\big)(\bar z)^T\lambda+\nabla \psi(\zb)^T\tau^*=0
\]
   and the second-order condition
  \begin{align*}
& \nabla^2 f(\bar{z})(d,d)-\langle \lambda, \nabla^2 \big( F + \nabla_y \psi^T\eta^*\big)(\bar z) (d,d)+2\big(\nabla (\nabla_y \psi)(\zb)d\big)^Te^*\rangle\\
&  +\skalp{\tau^*,\nabla^2\psi(\zb)(d,d)}\geq 0,
\end{align*}
where
$(\tau^*,\nabla_y \psi(\zb)\lambda ) \in N_{{\rm gph} N_P}\left ((\psi(\zb),\eta^*);(\nabla \psi(\zb)d,e^*)\right ).$
    \item[\rm (ii)] Suppose that for every nonzero $d$ satisfying \eqref{direction-bilevel} there is $(\alpha,\lambda,\tau^*)$ not all  equal to zero satisfying $\alpha\geq 0$ and
 \[
 \left\{
\begin{array}{l}
 \alpha\nabla f(\zb) -\nabla \big(F+\nabla_y \psi^T\eta^*\big)(\bar z)^T\lambda+\nabla \psi(\zb)^T\tau^*=0,\\
 (\tau^*,\nabla_y \psi(\zb)\lambda ) \in \widehat N_{T_{{\rm gph}N_P}(\psi(\zb),\eta^*)}(\nabla \psi(\zb)d,e^*)
\end{array}\right.
\]
  and
   \begin{align*}
 &\alpha\nabla^2 f(\bar{z})(d,d)-\langle \lambda, \nabla^2 \big(F + \nabla_y \psi^T\eta^*\big)(\bar z) (d,d)+2\big(\nabla (\nabla_y \psi)(\zb)d\big)^Te^*\rangle\\
 & +\skalp{\tau^*,\nabla^2\psi(\zb)(d,d)}>0,
\end{align*}
    then $\bar z$ is an essential local minimizer of second order for MPVI.
\end{itemize}
\end{theorem}

\begin{proof}
Note that $\Gamma(x)$ is convex since $\psi(x,y)$ is convex in $y$ and $P:=\R^l_-$.
 The nondegeneracy condition \eqref{nondege-condition} ensures $N_{\Gamma(x)}(y)=\nabla_y \psi(x,y)^TN_P(\psi(x,y))$. Hence (MPVI) can be reformulated as \eqref{MPVI-1} that is a special case of (GEP). Therefore, it only remains to show that the condition \eqref{nondege-condition} is equivalent to the condition \eqref{EqBasicAss}, which now takes the following formula
\begin{subequations}
  \begin{numcases}{}
    \nabla \psi(\bar z)^Td^*=0,\ \ d^* \in \Spanb{N_P(\psi(\bar z))}\ \Rightarrow d^*=0, \label{VI-1} \\
    \nabla_y \psi (\bar z)^Td^*=0,\ d^* \in \Spanb{N_P(\psi (\bar z))}\ \Rightarrow d^*=0. \label{VI-2}
     \end{numcases}
\end{subequations}
Notice that the condition \eqref{nondege-condition} is just \eqref{VI-2}. Meanwhile, \eqref{VI-1} is implied by \eqref{VI-2} since  $\nabla \psi(\bar z)^Td^*=(\nabla_x \psi(\bar z)^Td^*, \nabla_y \psi(\bar z)^Td^*)$. This establishes the equivalence between \eqref{nondege-condition} and \eqref{EqBasicAss}.
\end{proof}

\section{Concluding Remarks}
In this paper, we have derived second-order necessary and sufficient optimality conditions for very general problem GEP.
Our approach was to reformulate the GEP as a standard constrained optimization problem GP,
in which the set $\Omega$ is the image of the solution set of a disjunctive system under a smooth mapping.
We developed a comprehensive variational analysis of $\Omega$ and calculated the required second-order objects describing the curvature of the set $\Omega$.
As an illustration, the proposed framework was used to establish second-order optimality conditions for mathematical programs with variational inequality constraints (MPVIs).
Many other important problem classes can be modeled as GEPs, and we aim to apply these results to bilevel programs and mathematical programs with second-order cone complementarity constraints in future work.

\end{document}